\documentclass[12pt,a4paper]{amsart}
\usepackage{amssymb,xspace,epic,amsthm}
\usepackage{enumerate}
\usepackage[all]{xy}

\usepackage{amsmath,amssymb,xspace,amsthm}
\newtheorem{theorem}{Theorem}%[section]
\newtheorem{remark}[theorem]{Remark}%[section]
\newtheorem{lemma}[theorem]{Lemma}%[section]
\newtheorem{proposition}[theorem]{Proposition}%[section]
\newtheorem{corollary}[theorem]{Corollary}%[section]
\newtheorem{problem}[theorem]{Problem}%[section]
\numberwithin{equation}{section}

\renewcommand{\phi}{\varphi}

\renewcommand{\leq}{\leqslant}
\renewcommand{\geq}{\geqslant}

\begin{document}
\title[Combinatorial ideals for 
affine Lie algebras]{Combinatorial 
analogues of ad-nilpotent ideals for 
untwisted affine Lie algebras}
\author{Karin Baur and Volodymyr Mazorchuk}
\date{\today}

\begin{abstract}
We study certain types of ideals in the standard
Borel subalgebra of an untwisted affine Lie algebra.
We classify these ideals in terms of the root
combinatorics and give an explicit formula
for the number of such ideals in type $A$.
The formula involves various aspects of
combinatorics of Dyck paths and leads to a 
new interesting integral sequence.
\end{abstract}
\maketitle

\section{Introduction, motivation and description of the results}\label{s1}

Let $\mathfrak{g}$ be a simple finite dimensional 
complex Lie algebra with a fixed triangular decomposition
$\mathfrak{g}=\mathfrak{n}_-\oplus\mathfrak{h}
\oplus\mathfrak{n}_+$ and $\mathfrak{b}=
\mathfrak{h}\oplus\mathfrak{n}_+$ the corresponding Borel
subalgebra. An ideal $\mathfrak{i}$ of $\mathfrak{b}$ 
is called
{\em ad-nilpotent} if the adjoint action on $\mathfrak{b}$
of every element from $\mathfrak{i}$ is nilpotent.

It is easy to see that ad-nilpotent ideals of
$\mathfrak{b}$ are all contained in $\mathfrak{n}_+$.
Furthermore, as $\mathfrak{h}$ acts semi-simply on
$\mathfrak{n}_+$ with one-dimensional roots spaces,
any ad-nilpotent ideal $\mathfrak{i}$ decomposes into
a direct sum of (some of these) root spaces. It follows
that the number of such ideals is finite. Various
classification and enumeration problems related
to ad-nilpotent ideals in the above situation were
studied in \cite{CP1,CP2,CDR,KOP,AKOP,Pa}, see also references
therein.

The aim of the present paper is to generalize the
problem described above to the situation of 
affine Kac-Moody Lie algebras in a sensible and interesting way. 
There are several
natural obstructions which we 
have to deal with. So, let now
$\hat{\mathfrak{g}}=\hat{\mathfrak{n}}_-\oplus
\hat{\mathfrak{h}}
\oplus\hat{\mathfrak{n}}_+$ be an affine Kac-Moody Lie 
algebra with a fixed {\em standard} triangular
decomposition and $\hat{\mathfrak{b}}=
\hat{\mathfrak{h}}\oplus\hat{\mathfrak{n}}_+$ the 
corresponding Borel subalgebra.
First of all, it is easy to observe
that the only ad-nilpotent ideal of $\hat{\mathfrak{b}}$
is the zero ideal. Hence we propose instead to
consider {\em strong ideals} of $\hat{\mathfrak{b}}$,
that is ideals in $\hat{\mathfrak{b}}$ of finite codimension, 
which are contained in $\hat{\mathfrak{n}}_+$. For 
finite-dimensional  ${\mathfrak{g}}$ this notion coincides 
with that of an ad-nilpotent ideal of ${\mathfrak{b}}$. It is easy to see
that the number of strong ideals in $\hat{\mathfrak{b}}$
is infinite.

Let $\mathfrak{i}$ be a strong ideal in $\hat{\mathfrak{b}}$.
Then, similarly to the finite dimensional case, the 
ideal $\mathfrak{i}$ can be written as a direct sum of
its intersections with the root subspaces of $\hat{\mathfrak{g}}$.
However, unlike the finite dimensional case, in the case
of affine Lie algebras some root spaces of $\hat{\mathfrak{g}}$
might have dimension bigger than one. 
An ideal $\mathfrak{i}$ in $\hat{\mathfrak{b}}$ 
is called {\em thick} provided that the 
following condition is satisfied: for any root $\alpha$
of $\hat{\mathfrak{g}}$ the fact that $\mathfrak{i}\cap 
\hat{\mathfrak{g}}_{\alpha}\neq 0$ implies 
$\hat{\mathfrak{g}}_{\alpha}\subset \mathfrak{i}$. 
An ideal $\mathfrak{i}$ in $\hat{\mathfrak{b}}$ 
is called {\em combinatorial} provided that
it is strong and thick. Unfortunately,  the number of
combinatorial  ideals is still infinite.

Let $\mathfrak{i}$ be a combinatorial ideal. Then the set of
all roots $\alpha$ for which 
$\hat{\mathfrak{g}}_{\alpha}\subset \mathfrak{i}$ is called the
{\em support} of $\mathfrak{i}$ and is denoted by 
$\mathrm{supp}(\mathfrak{i})$ (the corresponding root space of 
$\mathfrak{i}$ is denoted $\mathfrak{i}_{\alpha}$). Let $\delta$ denote the
indivisible positive imaginary root of $\hat{\mathfrak{g}}$.
Two combinatorial ideals
$\mathfrak{i}$ and $\mathfrak{j}$ will be called
{\em equivalent} provided that there exists 
$k\in\mathbb{Z}$ such that 
$\mathrm{supp}(\mathfrak{i})=k\delta+
\mathrm{supp}(\mathfrak{j})$. We will show that 
for untwisted affine Lie algebras 
the number of equivalence classes of combinatorial ideals
in $\hat{\mathfrak{b}}$ is finite, which naturally leads
to the problems of their classification and enumeration.

We give a general answer to the classification problem
in the case of untwisted affine Lie algebras in terms
of certain root combinatorics. For the affine 
$\mathfrak{sl}_n$ we also answer the enumeration
problem in some nice combinatorial terms involving 
the combinatorics of Catalan numbers 
(via Dyck paths) and their
generalizations. For $n\in\mathbb{N}$ we will
define an $n\times n$-matrix $\mathbf{C}_n$ with
nonnegative integer entries (related to the
Dyck path combinatorics), and a linear
transformation $\omega$ on the linear space of
all $n\times n$-matrices (over some commutative ring)
such that for the standard scalar product 
$A\cdot B=\sum_{i,j}a_{i,j}b_{i,j}$ we have the
following:

\begin{theorem}\label{tmain}
The number of equivalence classes of combinatorial ideals
in the case $\hat{\mathfrak{g}}=\hat{\mathfrak{sl}}_n$
equals $\mathbf{b}_n:=\mathbf{C}_n\cdot \omega\,\mathbf{C}_n$. 
\end{theorem}

The integral sequence  
$\{\mathbf{b}_n:n\geq 1\}$
seems to be new. We note that appearance of Catalan
and related numbers in this problem is expected,
as Catalan numbers enumerate ad-nilpotent ideals
in the Borel subalgebra of $\mathfrak{sl}_n$
(a very natural bijection of such ideals with Dyck
paths can be found in \cite{Pa}, a less natural
bijection appears in \cite{AKOP}, however, the latter
one has the advantage that it controls the nilpotency
degree of the ideal in terms of some combinatorial
parameters of the corresponding Dyck path).

We show that the Dyck path combinatorics controls
various algebraic properties of combinatorial ideals
in the affine case as well. In particular, we
describe the number of generators for combinatorial ideals
in terms of  valleys and peaks of the corresponding Dyck
paths. We also propose a combinatorial generalization of the
notion of nilpotency degree, in particular, of that of an
abelian ideal (which we call quasi-abelian). We describe
quasi-abelian ideals combinatorially using intervals in
the poset of Dyck paths. Enumeration of quasi-abelian
combinatorial ideals for $\hat{\mathfrak{sl}}_n$ leads to
yet another new integral sequence. However, for this
one we do not have any explicit formula.

Finally, we also address the problem of studying arbitrary
$\hat{\mathfrak{b}}$-ideals in $\hat{\mathfrak{n}}_+$
and describe possible supports for such ideals. The above notion
of equivalence for ideals lifts to arbitrary ideals resulting
in what we call {\em support equivalent} ideals. This
leads to yet another finite classification problem for 
$\hat{\mathfrak{g}}$. In the case of $\hat{\mathfrak{sl}}_n$
we describe equivalence classes for the latter problem
combinatorially using quadruples of Dyck paths.

The paper is organized as follows: in Section~\ref{s2}
we consider the Lie theoretic part of the problem,
studying various types of ideals in the Borel subalgebra
of an untwisted affine Lie algebra. In the first
subsection we collect necessary preliminaries,
in the second subsection we introduce our main objects,
which we call combinatorial ideals, and in the last
subsection we classify equivalence classes of 
combinatorial ideals in terms of root combinatorics.
Section~\ref{s3} is mostly combinatorial. We start
by defining some constructions with matrices which
we use to develop some aspects of combinatorics of Dyck
paths. We then use these combinatorics to prove Theorem~\ref{tmain}
answering the enumeration problem for combinatorial ideals in type $A$.
In the last part of the section we relate various algebraic
properties of combinatorial ideals to combinatorics of Dyck paths.
In Section~\ref{s4} we consider all $\hat{\mathfrak{b}}$-ideals 
in $\hat{\mathfrak{n}}_+$ and describe their supports in terms
of root combinatorics, showing that this leads to a new finite problem.
In type $A$ we relate these combinatorics to that of Dyck paths.
\vspace{2mm}

\noindent
{\bf Acknowledgements.} The research was done during 
the visit of the first author to Uppsala University and 
the second author to FIM at ETH Zurich. The financial 
support and accommodation by the host universities are
gratefully acknowledged. The first author is supported 
by the Swiss National Science Foundation. The second 
author is partially supported by the Swedish Research 
Council and the Royal Swe\-dish Academy of  Sciences.
We thank Vyjayanthi Chari for bringing \cite{CDR}
to our attention.

\section{Combinatorial ideals for untwisted affine Lie algebras}\label{s2}

As usual, we denote by $\mathbb{N}$ and $\mathbb{N}_0$ 
the sets of positive and non-negative integers, 
respectively.

\subsection{Untwisted affine Lie algebras}\label{s2.1}

Let $\mathfrak{g}$ be a simple finite dimensional
complex Lie algebra and $\hat{\mathfrak{g}}$ the
corresponding untwisted affine Lie algebra. Then 
$\hat{\mathfrak{g}}$ is the universal central extension
of the extension of the loop algebra $\mathfrak{g}\otimes 
\mathbb{C}[t,t^{-1}]$, the Lie bracket in which is 
given by $[x\otimes t^k,y\otimes t^l]=[x,y]\otimes t^{k+l}$ for $x,y\in \mathfrak{g}$ and $k,l\in\mathbb{Z}$, 
by the derivation $d$ such that $[d,x\otimes t^k]=
k\, x\otimes t^k$ for $x\in \mathfrak{g}$ and 
$k\in\mathbb{Z}$, see \cite[Chapter~17]{Ca} for details.

Let $\mathfrak{g}=\mathfrak{n}_-\oplus
\mathfrak{h}\oplus\mathfrak{n}_+$ be a triangular
decomposition of $\mathfrak{g}$. Then the standard
triangular decomposition 
$\hat{\mathfrak{g}}=\hat{\mathfrak{n}}_-\oplus
\hat{\mathfrak{h}}\oplus\hat{\mathfrak{n}}_+$ is
defined as follows: $\hat{\mathfrak{h}}$ is spanned
by $\mathfrak{h}$, $d$ and the center; and
$\hat{\mathfrak{n}}_{\pm}$ is spanned by 
$\mathfrak{n}_{\pm}$
and all elements of the form $x\otimes t^{\pm k}$, 
where $x\in \mathfrak{g}$ and $k>0$. Set
$\hat{\mathfrak{b}}:=\hat{\mathfrak{h}}\oplus
\hat{\mathfrak{n}}_+$.

Let $\Delta\subset \mathfrak{h}^*$ and
$\hat{\Delta}\subset \hat{\mathfrak{h}}^*$ denote
the root system  of $\mathfrak{g}$ and $\hat{\mathfrak{g}}$,
respectively. Then the triangular decompositions above 
induce partitions $\Delta=\Delta_+\cup \Delta_-$
and $\hat{\Delta}=\hat{\Delta}_+\cup \hat{\Delta}_-$
of $\Delta$ and $\hat{\Delta}$, respectively, into
positive and negative roots. Abusing notation
we identify $\Delta$ with a subset of $\hat{\Delta}$
by viewing $\mathfrak{g}$ as a subalgebra of
$\hat{\mathfrak{g}}$ via the map $x\mapsto x\otimes 1$.

We denote by $\hat{\Sigma}_+$ (resp. ${\Sigma}_+$) 
the submonoid
of the additive monoid $\hat{\mathfrak{h}}^*$ 
(resp. ${\mathfrak{h}}^*$) generated
by $\hat{\Delta}_+$ (resp. ${\Delta}_+$).
Recall the usual partial order $\leq$ on
$\hat{\mathfrak{h}}^*$ defined as follows:
for $\lambda,\mu\in\hat{\mathfrak{h}}^*$ we have
$\lambda\leq \mu$ if and only if $\mu-\lambda\in
\hat{\Sigma}_+$. 
Then the simple roots in
$\hat{\Delta}_+$ are the simple roots of $\Delta$
together with the root $\alpha_0:=-\alpha_{\mathrm{max}}+\delta$,
where $\alpha_{\mathrm{max}}$ is the unique maximal
(with respect to $\leq$)
root in $\Delta_+$. 

We denote by $D$ the set of roots of all elements of the form $x\otimes 1$,
$x\in \mathfrak{n}_+$, and $x\otimes t$, $x\in \mathfrak{h}\oplus \mathfrak{n}_-$.
Then $\hat{\Delta}_+$ is the disjoint union of $D=D_0$ and
$D_k:=D+k\delta$ for $k> 0$. 
Further, for $\alpha,\beta\in D$ such that $\alpha\neq \beta$,
the difference $\alpha-\beta$ is not of the form
$k\delta$ for $k\in\mathbb{Z}$. Note that for
every ideal $\mathfrak{i}$ in $\hat{\mathfrak{n}}_+$
of finite codimension, the support of $\mathfrak{i}$
contains the union of all $D_k$ for $k\gg 0$.

\subsection{Basic ideals}\label{s2.2}

A combinatorial ideal $\mathfrak{i}$ of $\mathfrak{b}$ is
called {\em basic} provided that $\mathrm{supp}
(\mathfrak{i})\cap D\neq \varnothing$.

\begin{theorem}\label{thm1}
\begin{enumerate}[$($a$)$]
\item\label{thm1.0} We have $\displaystyle
\mathrm{supp}(\mathfrak{i})\supset D':=\bigcup_{k>0}D_k$ 
for any basic $\mathfrak{i}$.
\item\label{thm1.1} 
Every equivalence class of combinatorial ideals contains
a unique basic ideal.
\item\label{thm1.2} The number of equivalence classes of combinatorial ideals is finite.
\end{enumerate}
\end{theorem}

\begin{proof}
Let $\mathfrak{i}$ be a basic ideal. First we claim that
$\delta\in \mathrm{supp}(\mathfrak{i})$. Indeed, 
if $\mathrm{supp}(\mathfrak{i})$ contains some
$\alpha\in \Delta_+$, then, commuting with
positive root elements from $\mathfrak{n}_+$,
we get that $\mathrm{supp}(\mathfrak{i})$ contains
$\alpha_{\mathrm{max}}$. Commuting a nonzero root element
for the latter root with a nonzero root element for
the root $\alpha_0$ we get a nonzero element for
the root $\delta$, as required. If 
$\mathrm{supp}(\mathfrak{i})$ does not contain any
$\alpha\in \Delta_+$, then it either contains
$\delta$ (in which case we have nothing to prove)
or a root of the form $\alpha+\delta$ for some
$\alpha\in\Delta_-$. In the latter case, commuting
any nonzero root element for the root $\alpha+\delta$
with any nonzero root element for the positive
root $-\alpha$ we get a nonzero element for
the root $\delta$, as required.

As $\mathfrak{i}$ is thick, we then have 
$\mathfrak{g}_{\delta}\subset \mathfrak{i}$.
Commuting $\mathfrak{g}_{\delta}$ with 
$\hat{\mathfrak{n}}_{+}$ we obtain that 
$\mathrm{supp}(\mathfrak{i})$ contains all real
roots from $D'$. That  $\mathrm{supp}(\mathfrak{i})$
contains all $k\delta$ for $k>1$ follows from this
similarly to the previous paragraph. This proves
claim \eqref{thm1.0}.

Let $\mathfrak{i}$ be a combinatorial ideal. We claim that 
the equivalence class of $\mathfrak{i}$ contains a
basic ideal. We prove this by induction on the
minimal possible $k$ such that 
$\mathrm{supp}(\mathfrak{i})\cap D_k\neq \varnothing$.
If $k=0$, then $\mathfrak{i}$ is basic and we have 
nothing to prove. If $k>0$, define
\begin{displaymath}
\mathfrak{j}:=\bigoplus_{\alpha\in 
\mathrm{supp}(\mathfrak{i})} \hat{\mathfrak{g}}_{\alpha-\delta}.
\end{displaymath}
As $k>0$, the support of $\mathfrak{j}$ is contained
in $\cup_{i\geq k-1}D_i$ (and intersects $D_{k-1}$
non-trivially) and hence $\mathfrak{j}\subset 
\hat{\mathfrak{n}}_+$. As 
$\mathrm{supp}(\mathfrak{i})\supset \cup_{i\geq m}D_i$
for some $m$ big enough, the same is true for
$\mathrm{supp}(\mathfrak{j})$, which means that 
$\mathfrak{j}$ has finite codimension. 

Let $\alpha\in \mathrm{supp}(\mathfrak{j})$,
$x\otimes t^l\in \mathfrak{j}_{\alpha}$ 
for some $x\in\mathfrak{g}$ and $y\otimes t^m\in\hat{\mathfrak{n}}_+$
be a root element for a  root $\beta$. Assume that 
$[x\otimes t^l,y\otimes t^m]\neq 0$. Then 
$\alpha+\beta\in \hat{\Delta}_+$.
Moreover, $[x\otimes t^{l+1},y\otimes t^m]\neq 0$.
However, the root of $x\otimes t^{l+1}$ equals
$\alpha+\delta\in \mathrm{supp}(\mathfrak{i})$
and thus  $x\otimes t^{l+1}\in \mathfrak{i}$
as $\mathfrak{i}$ is thick. Therefore
$[x\otimes t^{l+1},y\otimes t^m]\neq 0$ also belongs
to $\mathfrak{i}$, which implies that 
$\alpha+\delta+\beta\in \mathrm{supp}(\mathfrak{i})$.
This yields that 
$\alpha+\beta\in \mathrm{supp}(\mathfrak{j})$.
Hence $[x\otimes t^l,y\otimes t^m]\in \mathfrak{j}$
by the definition of $\mathfrak{j}$. This shows 
that  $\mathfrak{j}$ is an ideal.

By construction, the ideals $\mathfrak{i}$ and
$\mathfrak{j}$ are equivalent. Moreover, as already 
mentioned above, the support of $\mathfrak{j}$
intersects $D_{k-1}$ non-trivially. Hence, by the 
inductive assumption, we get that the equivalence
class of $\mathfrak{i}$ contains a basic ideal.

Let now $\mathfrak{i}$ and $\mathfrak{j}$ be two equivalent
basic ideals. Then both $\mathrm{supp}(\mathfrak{i})$
and $\mathrm{supp}(\mathfrak{j})$ contain $D'$
by claim \eqref{thm1.0}. Further, 
$\mathrm{supp}(\mathfrak{i})\cap D=
\mathrm{supp}(\mathfrak{j})\cap D$ as for any
different $\alpha,\beta\in D$ the difference
$\alpha-\beta$ is not of the form $k\delta$, 
$k\in\mathbb{Z}$. This implies
$\mathrm{supp}(\mathfrak{i})=\mathrm{supp}(\mathfrak{j})$
and hence $\mathfrak{i}=\mathfrak{j}$ as both ideals
are thick, proving claim \eqref{thm1.1}.
 
Finally, from the above it follows that
a basic ideal is uniquely determined by the intersection
of its support with $D$. As $D$ is a finite set,
it follows that the number of basic ideals is finite.
Hence claim \eqref{thm1.2} follows from claim \eqref{thm1.1}.
\end{proof}

\begin{remark}
{\rm 
Using arguments similar to those used in the proof of
Theorem~\ref{thm1}\eqref{thm1.0} one can show that 
every non-zero ideal of $\hat{\mathfrak{n}}_+$ has
finite codimension, see Corollary~\ref{cor42}.
}
\end{remark}

\subsection{Classification of basic ideals}\label{s2.3}

For $\alpha\in D$ construct a subset $\overline{\alpha}$ 
of $D$ recursively as follows: set $\overline{\alpha}_0=
\{\alpha\}$, and for $i>0$ put
\begin{displaymath}
\overline{\alpha}_i=\{\gamma\in D:
\text{ there is }\alpha\in\hat{\Delta}_+
\text{ and }\beta\in \overline{\alpha}_{i-1}
\text{ such that }\gamma=\alpha+\beta\}.
\end{displaymath}
Then $\overline{\alpha}_0\subset \overline{\alpha}_1
\subset \overline{\alpha}_2\subset \dots$ by construction
and, as $D$ is finite, there is $i_0$ such that
$\overline{\alpha}_i=\overline{\alpha}_{i+1}$ for all
$i\geq i_0$. Set $\overline{\alpha}:=\overline{\alpha}_{i_0}$
(cf. \cite[Section~2]{CP1}).

Define the partial order $\preceq$ on $D$ as follows:
for $\alpha,\beta\in D$ set $\alpha\preceq\beta$ if and 
only if $\beta\in\overline{\alpha}$ (which is equivalent 
to $\overline{\beta}\subset \overline{\alpha}$). 
It is easy to see that $\delta$ is the unique maximal
element of $D$ with respect to $\preceq$.
For $\alpha\in D$ define 
\begin{displaymath}
\mathfrak{i}(\alpha):=
\bigoplus_{\beta\in\overline{\alpha}
\cup D'}\hat{\mathfrak{g}}_{\beta}.
\end{displaymath}

\begin{theorem}\label{thm2}
\begin{enumerate}[$($a$)$]
\item\label{thm2.1}  For any $\alpha\in D$ the space
$\mathfrak{i}(\alpha)$ is a basic ideal in 
$\hat{\mathfrak{b}}$.
\item\label{thm2.2} If $\mathfrak{j}$ is a basic ideal in 
$\hat{\mathfrak{b}}$ such that $\mathfrak{j}_{\alpha}\neq 0$
for some $\alpha\in D$,
then $\mathfrak{j}\supset \mathfrak{i}(\alpha)$.
\item\label{thm2.3} If $\mathfrak{j}$ is a basic ideal in 
$\hat{\mathfrak{b}}$, then 
\begin{displaymath}
\mathfrak{j}=\sum_{\alpha\in D\cap 
\mathrm{supp}(\mathfrak{j})} \mathfrak{i}(\alpha).
\end{displaymath}
\item\label{thm2.4} There is a bijection between the set
of basic ideals in $\hat{\mathfrak{b}}$ and nonempty
anti-chains of the finite poset $(D,\preceq)$.
\item\label{thm2.5} There is a bijection between the set
of basic ideals in $\hat{\mathfrak{b}}$ and 
submodules of the $\hat{\mathfrak{b}}$-module
$\hat{\mathfrak{n}}_+/\mathfrak{i}(\delta)$.
\end{enumerate}
\end{theorem}

\begin{proof}
Claim \eqref{thm2.1} follows directly from the definitions. 
To prove claim \eqref{thm2.2}, let  $\mathfrak{j}$ be a 
basic ideal in $\hat{\mathfrak{b}}$ such that $\mathfrak{j}_{\alpha}\neq 0$. Then $D'\subset
\mathrm{supp}(\mathfrak{j})$ by 
Theorem~\ref{thm1}\eqref{thm1.0}. Let
$\beta\in D$ be such that $\alpha\preceq\beta$.
To prove claim \eqref{thm2.2} it is enough to show that
$\beta\in \mathrm{supp}(\mathfrak{j})$. If $\beta=\delta$,
then $\beta\in \mathrm{supp}(\mathfrak{j})$ as was shown 
in the proof of Theorem~\ref{thm1}\eqref{thm1.0}, hence
we may assume $\beta\neq \delta$ (i.e. $\beta$ is a real 
root). As $\alpha\preceq\beta$, there is a sequence of
real roots $\gamma_0=\alpha,\gamma_1,\dots,\gamma_k=\beta$
in $D$ such that $\xi_i:=
\gamma_i-\gamma_{i-1}\in\hat{\Delta}_+$
for all $i=1,2,\dots,k$. Note that each $\xi_i$ is
real and hence gives rise to an $\mathfrak{sl}_2$-subalgebra
$\mathfrak{s}_i$ of $\hat{\mathfrak{g}}$. Then $\dim \hat{\mathfrak{g}}_{\gamma_i}=1$
for all $i$ as all $\gamma_i$ are real, and the classical
$\mathfrak{sl}_2$-theory applied to the adjoint action
of $\mathfrak{s}_i$ on $\hat{\mathfrak{g}}$ implies 
$[\hat{\mathfrak{g}}_{\xi_i},\hat{\mathfrak{g}}_{\gamma_{i-1}}]\neq 0$. The latter
yields $\beta\in \mathrm{supp}(\mathfrak{j})$ and
claim \eqref{thm2.2} follows.

Claim \eqref{thm2.3} follows directly from claim
\eqref{thm2.2} and Theorem~\ref{thm1}\eqref{thm1.0}. 
To prove claim \eqref{thm2.4}, let $\mathfrak{j}$
be a basic ideal. Denote by $B_{\mathfrak{j}}$ the
set of all minimal (with respect to $\preceq$)
elements in $D\cap \mathrm{supp}(\mathfrak{j})$.
Then $B_{\mathfrak{j}}$ is a nonempty anti-chain in $D$.
Conversely, given a nonempty anti-chain $B$ in $D$,
define $\overline{B}$ to be the coideal of $D$
generated by $B$ and set
\begin{displaymath}
\mathfrak{i}_{B}:=
\bigoplus_{\beta\in\overline{B}
\cup D'}\mathfrak{g}_{\beta}.
\end{displaymath}
Then it is easy to check that $\mathfrak{i}_{B}$
is a basic ideal and that the maps $\mathfrak{i}\mapsto
B_{\mathfrak{i}}$ and $B\mapsto \mathfrak{i}_{B}$
are mutually inverse bijections. Claim \eqref{thm2.4} follows.

If $\mathfrak{i}$ is a basic ideal, then 
$\mathfrak{i}\supset \mathfrak{i}(\delta)$ and hence
the image of $\mathfrak{i}$ in the 
$\hat{\mathfrak{b}}$-module
$\hat{\mathfrak{n}}_+/\mathfrak{i}(\delta)$
is a submodule and this map from the set of basic ideals
to the set of submodules in 
$\hat{\mathfrak{n}}_+/\mathfrak{i}(\delta)$
is injective. It is also easily seen to be surjective
as the full preimage of a submodule in
$\hat{\mathfrak{n}}_+/\mathfrak{i}(\delta)$ is a
basic ideal (here it is important that all 
$\hat{\mathfrak{h}}$-weight spaces of 
$\hat{\mathfrak{n}}_+/\mathfrak{i}(\delta)$ are
one-dimensional as they correspond to real roots).
This completes the proof.
\end{proof}

The ideal $\mathfrak{i}(\alpha)$ can be considered
as a kind of ``principal'' basic ideal generated by
$\hat{\mathfrak{g}}_{\alpha}$. The next proposition relates
the natural order $\leq$ to the order $\preceq$ on $D$ defined above.

\begin{proposition}\label{prop4}
The orders $\leq$ and $\preceq$ coincide on $D$. 
\end{proposition}

\begin{proof}
Obviously, $\preceq$ is a subset of $\leq$ (as a binary 
relation), so we only have to prove the reverse inclusion.
Let $\alpha,\beta\in D$ be such that $\alpha\leq\beta$. 
We have to show that $\alpha\preceq\beta$. If $\beta=\delta$,
then $\alpha\preceq\delta$ is clear, so in the following we
may assume that $\beta\neq \delta$, that is that 
both $\alpha$ and $\beta$ are real roots.

{\em Case~1.} Assume first that $\beta\in\Delta_+$, then
$\alpha\in\Delta_+$ as well. Let $\gamma_1,\dots,\gamma_k$
be simple roots such that $\beta-\alpha\in\sum_{i=1}^k
\mathbb{N}\,\gamma_i$ and $\mathfrak{a}$ be the semi-simple
Lie subalgebra of $\hat{\mathfrak{g}}$ which these roots
(and their negatives) generate. Consider the 
finite-dimensional $\mathfrak{a}$-module
\begin{displaymath}
V:=\bigoplus_{\xi\in \alpha+
\sum_{i=1}^k\mathbb{Z}\gamma_i} 
\hat{\mathfrak{g}}_{\xi}.
\end{displaymath}
It is enough to show that $\alpha+\gamma_i$ is a 
weight of $V$ for some $i$. Indeed, if this is the case,
then $\alpha\preceq \alpha+\gamma_i$, furthermore,
$\alpha+\gamma_i\leq \beta$ and the claim follows
by induction on the height of $\beta-\alpha$.
Assume that $\alpha+\gamma_i$ is not a 
weight of $V$ for every $i$. Then $\alpha$ is an
$\mathfrak{a}$-highest
weight of $V$. Let $\beta'\geq \beta$ be the highest 
weight of the unique simple subquotient $V'$ of $V$, which
intersects the one-dimensional space 
$\hat{\mathfrak{g}}_{\beta}$ non-trivially. Then 
$\alpha\leq \beta'$ are two $\mathfrak{a}$-dominant
weights. By \cite[Proposition~21.3]{Hu}, $\alpha$
is a weight of $V'$.
Since $\hat{\mathfrak{g}}_{\alpha}$ is one-dimensional,
and $\alpha$ is a highest weight of $V$, we get
$\alpha=\beta'=\beta$ and we are done.

{\em Case~2.} Assume that $\alpha\not\in\Delta_+$, then
$\beta\not\in\Delta_+$ as well. Then $\alpha\leq \beta$
implies $-\beta+\delta\leq -\alpha+\delta$ and
$-\beta+\delta,-\alpha+\delta\in\Delta_+$. From
Case~1 we have $-\beta+\delta\preceq -\alpha+\delta$,
which implies $\alpha\preceq\beta$.

{\em Case~3.} Finally, assume that $\alpha\in\Delta_+$
while $\beta\not\in\Delta_+$. In this case we will need
the following auxiliary lemma:

\begin{lemma}\label{lem6}
There does not exist a decomposition of the maximal
root of $\Delta_+$ of the form
\begin{equation}\label{eq2}
\alpha_{\mathrm{max}}=\xi+\zeta+\eta
\end{equation}
such that the
following conditions are satisfied:
\begin{enumerate}[$($i$)$]
\item\label{lem6.1} $\xi,\zeta\in \Delta_{+}$
while $\xi+\zeta\not\in \Delta_{+}$;
\item\label{lem6.2} $\eta$ is a linear
combination of some simple roots 
$\eta_1,\dots,\eta_k$, $k\geq 1$,
with positive integer coefficients;
\item\label{lem6.3} $\xi+\eta_i\not\in \Delta_{+}$ and
$\zeta+\eta_i\not\in \Delta_{+}$ for all $i$.
\end{enumerate}
\end{lemma}

Given Lemma~\ref{lem6}, the proof of Case~3 goes
as follows: As $\alpha\in\Delta_+$, 
$\beta\not\in\Delta_+$
and $\alpha\leq\beta$,
we can write $\beta=\alpha+\alpha_0+\gamma$
for some $\gamma\in \Sigma_+$. Let $\gamma_1,\dots,\gamma_k$
denote all simple roots which appear in the decomposition
of $\gamma$ as a linear combination of simple roots
with positive integer coefficients. If $\alpha+\gamma_i$
is a root for some $i$, then $\alpha\preceq\alpha+\gamma_i$
and $\alpha+\gamma_i\leq \beta$ by construction, so we
can complete the argument by induction on the height of
$\beta-\alpha$. Similarly, if $\beta-\gamma_i$ is a
root for some $i$. Finally, if 
$\alpha-\beta+\delta\in\Delta_+$, 
then $\beta-\alpha\in\hat{\Delta}_+$
and hence $\alpha\preceq\beta$ by definition. If none
of the above is satisfied, then, taking
$\xi=\alpha$, $\zeta=-\beta+\delta$ and $\eta=\gamma$, we
see that these elements satisfy conditions
\eqref{lem6.1}--\eqref{lem6.3} of Lemma~\ref{lem6},
which is a contradiction by Lemma~\ref{lem6}.
This completes the proof.
\end{proof}

\begin{proof}[Proof of Lemma~\ref{lem6}.]
We prove Lemma~\ref{lem6} by a brute force case-by-case
analysis of all reduced irreducible finite root systems.
Before we start this analysis we make some remarks. 
Assume that the decomposition of the form \eqref{eq2}
exists. Let $\Delta_{\eta}$ denote the root system
generated by $\eta_i$, $i=1,\dots,k$. We have
$\varnothing\neq \Delta_{\eta}\subsetneq \Delta$
(the latter inequality follows from conditions 
\eqref{lem6.1} and \eqref{lem6.3}).

From condition \eqref{lem6.3} it follows that
reflection with respect to each $\eta_i$ either
does not effect $\xi$ and $\zeta$ or decreases them
(in this proof by ``decreases'' we always mean
``with respect to $\leq$''). The element $\eta$
satisfies $0< \eta<\alpha_{\mathrm{max}}$
and thus must be decreased by the
reflection with respect to at least one of the $\eta_i$'s.
This and \eqref{eq2} imply that every decrease of 
$\eta$ automatically gives a decrease of 
$\alpha_{\mathrm{max}}$. The coordinate of
a vector $v\in\mathfrak{h}^*$ (with distinguished basis
of simple roots), reflection with 
respect to which  decreases $v$, will be called
{\em decreasable}. We will also say that the
corresponding simple root {\em decreases} $v$.

{\em Type~$A$.} In this case $\alpha_{\mathrm{max}}$
is as follows, with decreasable coordinates in bold
(here and in the rest of the proof we use 
\cite[Chapter~III]{Hu} as a reference):
\begin{displaymath}
\xymatrix{ 
{\bf 1}\ar@{-}[r]&1\ar@{-}[r]&1\ar@{-}[r]
&\dots\ar@{-}[r]&1\ar@{-}[r]&1\ar@{-}[r]&{\bf 1}
}
\end{displaymath}
The system $\Delta_{\eta}$ must contain at least one simple root
which decreases $\alpha_{\mathrm{max}}$.
Without loss of generality we thus may assume that 
$\Delta_{\eta}$ contains the leftmost simple
root. Then $\eta$'s coordinate at it must be $1$
and, since this coordinate is decreasable, the next
from the left coordinate must be $1$ as well.
Since this coordinate is not decreasable, its right
neighbor must be $1$ again and so on. We get
$\eta=\alpha_{\mathrm{max}}$, which is not possible.

{\em Type~$D$.} In this case $\alpha_{\mathrm{max}}$
is as follows, with the unique decreasable coordinate 
in bold:
\begin{displaymath}
\xymatrix{ 
{1}\ar@{-}[r]&{\bf 2}\ar@{-}[r]&2\ar@{-}[r]
&\dots\ar@{-}[r]&2\ar@{-}[r]&2\ar@{-}[r]\ar@{-}[d]&{1}\\
&&&&&1&
}
\end{displaymath}
The system $\Delta_{\eta}$ must contain the simple root 
which decreases $\alpha_{\mathrm{max}}$. The corresponding
coordinate of $\eta$ can thus be $1$ or $2$. If we
assume that it is $2$, then, using the uniqueness of
decreasable coordinate in $\eta$ one shows that
$\eta=\alpha_{\mathrm{max}}$, a contradiction.
Hence this coordinate is $1$. The coordinate to the
left cannot be $1$ as in this case it would be decreasable,
hence it is $0$. This forces the coordinate to the right
to be $1$. The latter is not decreasable which forces the
next coordinate to the right to be  $1$ or $2$. But a
$2$ gives a decreasable coordinate, and hence the only
possibility is $1$. In this way we show that all
nonzero coordinates of $\eta$ have to be $1$ and thus
the last nonzero coordinate which we get will be 
decreasable, a contradiction.

{\em Type~$B$.} 
In this case $\alpha_{\mathrm{max}}$
is as follows, with the unique decreasable coordinate 
in bold:
\begin{displaymath}
\xymatrix{ 
{1}\ar@{-}[r]&{\bf 2}\ar@{-}[r]&2\ar@{-}[r]
&\dots\ar@{-}[r]&2\ar@{-}[r]&2\ar@{=>}[r]&{2}
}
\end{displaymath}
Here the argument is similar to the one we used in type $D$.

{\em Type~$C$.} In this case $\alpha_{\mathrm{max}}$
is as follows, with the unique decreasable coordinate 
in bold:
\begin{displaymath}
\xymatrix{ 
{\bf 2}\ar@{-}[r]&{2}\ar@{-}[r]&2\ar@{-}[r]
&\dots\ar@{-}[r]&2\ar@{-}[r]&2&{1}\ar@{=>}[l]
}
\end{displaymath}
Here again the argument is similar to the one we used 
in type $D$ with the difference that we do not need
to bother about what happens to the left of the
decreasable root.

{\em Type~$E_6$.} In this case $\alpha_{\mathrm{max}}$
is as follows, with the unique decreasable coordinate 
in bold:
\begin{displaymath}
\xymatrix{ 
{1}\ar@{-}[r]&{2}\ar@{-}[r]&3\ar@{-}[r]\ar@{-}[d]
&2\ar@{-}[r]&{1}\\
&&{\bf 2}&&
}
\end{displaymath}
Similarly to the previous cases we get that the
coordinate of $\eta$ at the unique decreasable
root of $\alpha_{\mathrm{max}}$ cannot be $2$ for
the latter forces $\eta=\alpha_{\mathrm{max}}$,
so this coordinate is $1$. So the coordinate of
$\eta$ in its unique neighbor must be also $1$.
Now the next neighbor to the right can have 
coordinates $0$, $1$ or $2$. The value $2$
is not possible as the coordinate will be decreasing.
The value $1$ forces the value $1$ in the next
right neighbor  by the ``type $A$''-argument,
and this last coordinate becomes decreasing. 
Hence the only possibility is $0$. By symmetry,
the value of the coordinate to the left from
the triple point is also zero. This forces the
triple point to be decreasing and we again have a contradiction.

{\em Types~$E_{7,8}$.} Are similar to $E_6$ 
(but with many more cases to go through) and left
to the reader.

{\em Type~$F_{4}$.} In this case $\alpha_{\mathrm{max}}$
is as follows, with the unique decreasable coordinate 
in bold:
\begin{displaymath}
\xymatrix{ 
{\bf 2}\ar@{-}[r]&{3}\ar@{=>}[r]&4\ar@{-}[r]
&2
}
\end{displaymath}
Similarly to the previous cases we see that the
leftmost coordinate of $\eta$ must have value $1$,
which also forces value $1$ or its right neighbor.
Now the third coordinate cannot be decreasing and hence
can only have values $1,2$. The value $2$ forces
the rightmost value to be $2$ and this rightmost 
coordinate becomes decreasing, a contradiction.
Hence the third coordinate has value $1$. The
last coordinate must then be zero as 
$\Delta\neq \Delta_{\eta}$. But so far we could have
\begin{displaymath}
\eta\quad=\quad\xymatrix{ 
{1}\ar@{-}[r]&{1}\ar@{=>}[r]&1\ar@{-}[r].
&0
}
\end{displaymath}
However, in this case $\xi+\zeta$ is a positive root
which contradicts condition \eqref{lem6.1}.

{\em Type~$G_{2}$.} This is a short direct computation
which is left to the reader.
\end{proof}

\begin{remark}\label{rem5}
{\rm
The special case of Proposition~\ref{prop4} in the classical
case of finite dimensional $\mathfrak{g}$ is implicit when
comparing the first paragraph of \cite[Section~2]{CP1}
with the first paragraph of \cite[Section~2]{AKOP}.
However, we did not manage to find any explicit proof in 
the literature. An alternative proof of this special
case follows, for example, from \cite[Lemma~1.1(ii)]{CDR}.
}
\end{remark}

Let $\mathfrak{i}$ be a basic ideal. Then
$\mathfrak{i}_+:=\mathfrak{i}\cap \mathfrak{n}_+$
is an ad-nilpotent ideal in $\mathfrak{b}$
(or, equivalently, a $\mathfrak{b}$-submodule
of $\mathfrak{n}_+$). Consider the $\mathfrak{b}$-module
$\mathfrak{g}/\mathfrak{b}$ which is identified, both
as a vector space and an $\mathfrak{h}$-module,
with $\mathfrak{n}_-$ in the natural way.
Denote by $\mathfrak{i}_-$ the $\mathfrak{b}$-submodule
of $\mathfrak{g}/\mathfrak{b}$ which under this
identification corresponds to the direct sum of root
spaces of negative roots $\beta$ such that $\beta+\delta\in
\mathrm{supp}(\mathfrak{i})$. From Theorem~\ref{thm2}
it follows that the pair $(\mathfrak{i}_+,\mathfrak{i}_-)$
determines $\mathfrak{i}$ uniquely.

Denote by $\mathfrak{b}_-$ the opposite Borel
subalgebra $\mathfrak{h}\oplus\mathfrak{n}_-$.
It is easy to see that the $\mathfrak{h}$-complement
of $\mathfrak{i}_-$ in $\mathfrak{n}_-$ is in fact
a $\mathfrak{b}_-$-submodule of $\mathfrak{n}_-$.
Applying the Chevalley involution we obtain that
$\mathfrak{i}$ is uniquely determined by a pair
of ad-nilpotent ideals in $\mathfrak{b}$. 
This implies the following:

\begin{corollary}\label{cor21}
The number of basic ideals in $\hat{\mathfrak{b}}_+$
does not exceed the square of the number of
ad-nilpotent ideals in $\mathfrak{b}$.
\end{corollary}

Later on we will see that not every pair of
ad-nilpotent ideals in $\mathfrak{b}$ corresponds
in this way to a basic ideals in $\hat{\mathfrak{b}}_+$.
In the next section we exploit this connection
to enumerate basic ideals in $\hat{\mathfrak{b}}_+$
for the affine Lie algebra $\hat{\mathfrak{sl}}_n$.

\section{Enumeration of basic ideals in type $A$}\label{s3}

\subsection{Some matrix combinatorics}\label{s3.1}

Let $\Bbbk$ be a commutative ring with $1$. For
$n\in\mathbb{N}$ consider the $\Bbbk$-algebra
$\mathrm{M}_n(\Bbbk)$ of all $n\times n$ matrices
with coefficients from $\Bbbk$. We will denote 
elements of $\mathrm{M}_n(\Bbbk)$ in the following
standard way:
\begin{displaymath}
A=\left(A_{i,j}\right)=\left(a_{i,j}\right)=
\left(\begin{array}{cccc}
a_{1,1}&a_{1,2}&\dots&a_{1,n}\\
a_{2,1}&a_{2,2}&\dots&a_{2,n}\\
\vdots&\vdots&\ddots&\vdots\\
a_{n,1}&a_{n,2}&\dots&a_{n,n}\\
\end{array}\right)
\end{displaymath}
and for such notation the indices $i$ and
$j$ always run through the set $\{1,2,\dots,n\}$.

For $A,B\in \mathrm{M}_n(\Bbbk)$ set
$A\cdot B=\sum_{i,j}a_{i,j}b_{i,j}$. Then 
$\cdot:\mathrm{M}_n(\Bbbk)\times\mathrm{M}_n(\Bbbk) 
\to \Bbbk$ is a bilinear form on $\mathrm{M}_n(\Bbbk)$.
Define the linear operator $\omega$ on
$\mathrm{M}_n(\Bbbk)$ by setting
\begin{displaymath}
(\omega\,A)_{i,j}=
\sum_{k=n-i}^n\,\,\sum_{m=n-j}^n a_{k,m}
\end{displaymath}
(here we assume that $a_{0,m}=a_{k,0}=0$).
For example, we have
\begin{displaymath}
\omega\,\left(\begin{array}{ccc}1&2&3\\4&5&6\\7&8&9\end{array}\right)=
\left(\begin{array}{ccc}28&33&33\\
39&45&45\\39&45&45\end{array}\right).
\end{displaymath}

We will also need another linear operator 
$\tau$ on $\mathrm{M}_n(\Bbbk)$ defined as follows:
\begin{equation}\label{eq5}
(\tau\,A)_{i,j}=
\sum_{s=i-1}^n a_{s,j}
\end{equation}
(here we assume $a_{0,j}=0$). For example, 
we have
\begin{displaymath}
\tau\,\left(\begin{array}{ccc}1&2&3\\4&5&6\\
7&8&9\end{array}\right)=
\left(\begin{array}{ccc}12&15&18\\12&15&18\\
11&13&15\end{array}\right).
\end{displaymath}

\subsection{Some Dyck path combinatorics}\label{s3.2}

In this subsection we collect some necessary elementary
combinatorics of Dyck paths (most of which we failed
to find in the literature, but we are not going to be 
surprised if it exists). For $n\in\mathbb{N}$ define 
the integral matrix $\mathbf{C}_n\in
\mathrm{M}_n(\mathbb{Z})$ recursively as follows:
\begin{equation}\label{eq4}
\mathbf{C}_1:=(1);\quad
\mathbf{C}_{n+1}=\big(\tau \, \mathbf{C}_{n}\big)
\oplus \mathbf{C}_1.
\end{equation}
For small values of $n$ we have:
\begin{gather*}
\mathbf{C}_1:=\left(\begin{array}{c}1\end{array}\right);\quad
\mathbf{C}_2:=\left(\begin{array}{cc}
1&0\\0&1\end{array}\right);\quad
\mathbf{C}_3:=\left(\begin{array}{ccc}
1&1&0\\1&1&0\\0&0&1\end{array}\right);\\
\mathbf{C}_4:=\left(\begin{array}{cccc}
2&2&1&0\\2&2&1&0\\1&1&1&0\\0&0&0&1\end{array}\right);\quad
\mathbf{C}_5:=\left(\begin{array}{ccccc}
5&5&3&1&0\\5&5&3&1&0\\3&3&2&1&0\\1&1&1&1&0\\
0&0&0&0&1\end{array}\right).
\end{gather*}
From the definition it is clear that all
$\mathbf{C}_n$ have non-negative coefficients.
What is much less clear, but suggested by the above
examples, is that every matrix $\mathbf{C}_n$ is symmetric.
We will prove this later in Corollary~\ref{cor8}. 
We denote the coefficient
$(\mathbf{C}_n)_{i,j}$ by $\mathbf{c}_{i,j}^{(n)}$.

Recall that for $n\in\mathbb{N}_0$ a {\em Dyck path}
of semilength $n$ is a path in the first quadrant of the
coordinate $(x,y)$-plane from $(0,0)$ to $(2n,0)$ with steps
along  $(1,1)$, called {\em rises}, and 
along  $(1,-1)$, called {\em falls}. A {\em peak} of
a Dyck path is the end point of a rise followed by a fall.
A {\em valley} of a Dyck path is the end point of a 
fall followed by a rise. The $y$-coordinate of a
point is usually called the {\em height}. A Dyck path
is called {\em primitive} if $(0,0)$ and $(2n,0)$
are the only points from the $x$-axis which belong to the
path. For a Dyck path $p$ we denote by $\mathbf{v}(p)$
and $\mathbf{p}(p)$ the number of valleys and peaks
in $p$, respectively. By $\mathbf{p}_i(p)$, $i=1,\dots,n$,
we also denote the number of peaks of height at least
$i+1$. By $\mathbf{v}_{(i)}(p)$, $i=0,\dots,n-1$,
we denote the number of valleys of height $i$ in $p$. 
By $\mathtt{v}_{(i)}(p)$, $i=0,\dots,n-1$,
we denote the set of valleys of height $i$ in $p$, in particular,
$\mathbf{v}_{(i)}(p)=|\mathtt{v}_{(i)}(p)|$. Rises, falls,
peaks and valleys are counted from the left to the right,
e.g. the first peak is the leftmost peak and
the last valley is the rightmost valley. Note that
every Dyck path has at least one peak, but it may
contain no valleys. Further, every Dyck path starts
with a rise and ends with a fall.
We denote by $\mathcal{D}_n$
the set of all Dyck paths of semilength $n$.
It is well-known (see for example \cite[Chapter~6]{St})
that $|\mathcal{D}_n|$ is the $n$-th Catalan
number $C_n=\frac{1}{n+1}\binom{2n}{n}$.
The set $\mathcal{D}_n$ is partially ordered 
in the natural way with 
respect to the relation $p\leq q$ defined as
follows: path $q$ never goes below path $p$.

A natural way to encode Dyck paths is by the corresponding 
sequence of rises and falls, e.g. $rrfrffrf$
(the condition is that the sequence contains $n$ rises
and $n$ falls and each prefix of the sequence contains
at least as many rises as falls). See an example
in Figure~\ref{fig1}. We denote by $*$ the usual 
involution on Dyck paths defined by reversing them 
and swapping rises and  falls (which corresponds to 
reflecting the Dyck path at the vertical line given by the
equation $x=n$), for example
\begin{displaymath}
(rrrfrrffff)^*=rrrrffrfff. 
\end{displaymath}
We denote by $\mathbf{p}$ the unique Dyck paths with no valleys,
that is the path of the form $rr\dots rff\dots f$.
We also denote by $\mathbf{q}$ the unique Dyck paths with $n$ peaks,
that is the path of the form $rfrf\dots rf$. The paths 
$\mathbf{p}$ and $\mathbf{q}$ are the maximum and the minimum elements
of $\mathcal{D}_n$ with respect to the partial order $\leq$, respectively.

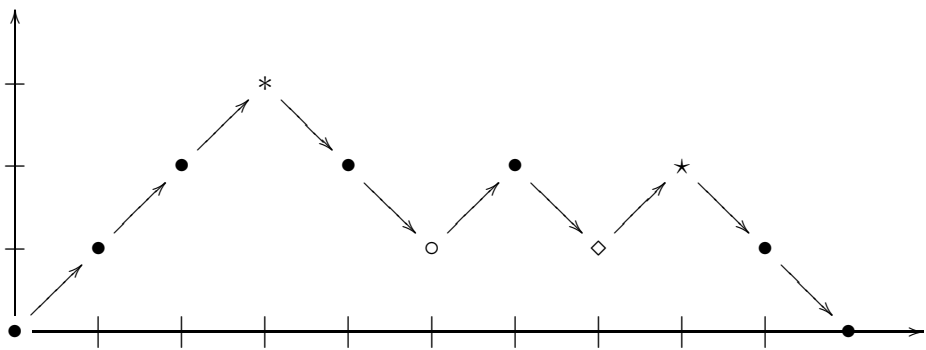
\begin{figure}
\begin{displaymath}
\xymatrix@!=0.6pc{
&&&&&&&&&\\
-&&&\ast\ar[rd]&&&&&&&\\
-&&\bullet\ar[ru]&&\bullet\ar[rd]&&
\bullet\ar[rd]&&\star\ar[rd]&&&\\
-&\bullet\ar[ru]&&&&\circ\ar[ru]&&
\diamond\ar[ru]&&\bullet\ar[rd]&&\\
\bullet\ar[ru]\ar[rrrrrrrrrrr]\ar[uuuu]
&\vert&\vert&\vert&\vert&\vert&
\vert&\vert&\vert&\vert&\bullet&
}
\end{displaymath}
\caption{The Dyck path 
$rrrffrfrff\in \mathcal{D}_5(3,2)$, the first peak
is $\ast$ (of height $3$), the last peak is $\star$
(of height $2$), the first valley
is $\circ$, the last valley is $\diamond$
(both valleys have height $1$)}\label{fig1} 
\end{figure}

The height of a peak of a Dyck path of semilength 
$n$ is a positive integer between $1$ and $n$.
For $i,j\in\{1,2,\dots,n\}$ denote by $\mathcal{D}_n(i,j)$
the set of all Dyck paths for which $i$ is the height
of the first peak and $j$ is the height of the
last peak (note that some of the $\mathcal{D}_n(i,j)$
might be empty). We also set 
$\mathcal{D}_n(0,j)=\varnothing$ for all $j$.

\begin{theorem}\label{thm7}
For all $n\in\mathbb{N}$ and $i,j\in\{1,2,\dots,n\}$ 
we have 
\begin{displaymath}
|\mathcal{D}_n(i,j)|=\mathbf{c}_{i,j}^{(n)}. 
\end{displaymath}
\end{theorem}

\begin{proof}
We proceed by induction on $n$. The basis $n=1$
is trivial. Further, note that
$\mathcal{D}_n(n,j)$ is non-empty only in the case
$j=n$, moreover $\mathcal{D}_n(n,n)=\{\mathbf{p}\}$. 
Similarly for $\mathcal{D}_n(i,n)$. Hence the statement of 
our theorem is true for all $n$ in all cases where $
i=n$ or $j=n$ as by construction of $\mathbf{C}_n$
we have $\mathbf{c}_{n,n}^{(n)}=1$ and $\mathbf{c}_{i,n}^{(n)}=0=
\mathbf{c}_{n,j}^{(n)}$ for all $i,j\neq n$.

Define the map
\begin{displaymath}
F: \mathcal{D}_n(i,j)\to \bigcup_{s=i-1}^{n-1}
\mathcal{D}_{n-1}(s,j)
\end{displaymath}
as follows: given a Dyck path in $\mathcal{D}_n(i,j)$,
delete the first peak in this Dyck path,
that is the leftmost occurrence of $rf$.
Define the map
\begin{displaymath}
G: \bigcup_{s=i-1}^{n-1}
\mathcal{D}_{n-1}(s,j)\to\mathcal{D}_n(i,j)
\end{displaymath} 
as follows: given a Dyck path in 
$\cup_{s=i-1}^{n-1} \mathcal{D}_{n-1}(s,j)$,
insert a peak after the first $i-1$ rises,
that is an $rf$ after the first $i-1$ letters
$r$ of the path. It is straightforward to verify
that $F$ and $G$ are mutually inverse bijections,
and hence
\begin{equation}\label{eq3}
\left|\bigcup_{s=i-1}^{n-1}
\mathcal{D}_{n-1}(s,j)\right|=
\left|\mathcal{D}_n(i,j)\right|
\end{equation} 
by the bijection rule. As the union on the left hand side
is disjoint, from the inductive assumption we have
\begin{displaymath}
\begin{array}{rcl}
\left|\mathcal{D}_n(i,j)\right|&
\overset{\eqref{eq3}}{=}&\displaystyle
\left|\bigcup_{s=i-1}^{n-1}
\mathcal{D}_{n-1}(s,j)\right|\\
&\overset{\text{(by induction)}}{=}&\displaystyle
\sum_{s=i-1}^{n-1}\mathbf{c}_{s,j}^{(n-1)}\\
&\overset{\eqref{eq4}\text{ and }\eqref{eq5}}{=}&
\mathbf{c}_{i,j}^{(n)}.
\end{array}
\end{displaymath} 
The proof is complete.
\end{proof}

Various enumeration problems involving Dyck paths
were considered in \cite{De}, in particular, there
one can find a formula for enumeration with 
respect to the height of the first peak.

\begin{corollary}\label{cor8}
The matrix $\mathbf{C}_n$ is symmetric
for every $n\in\mathbb{N}$ and the sum of all
entries in $\mathbf{C}_n$ equals $C_n$.
\end{corollary}

\begin{proof}
The second claim follows directly from
Theorem~\ref{thm7}. The first claim  
follows from Theorem~\ref{thm7} applying the
involution $*$.
\end{proof}

The entries of $\mathbf{C}_n$ are directly related
to several classical combinatorial objects associated
with the combinatorics of Catalan numbers. Our first
observation is a relation to what is known as the
{\em Catalan triangle} (sequence A009766 in
\cite{OEIS}), see \cite{Ba}. It is defined, in analogy 
with Pascal's triangle, as the following triangular array
of integers with the property that every entry equals
the sum of the entry above and the entry to the left:
\begin{displaymath}
\begin{array}{cccccc}
1&&&&&\\ 
1&1&&&&\\ 
1&2&2&&&\\ 
1&3&5&5&&\\ 
1&4&9&14&14&\\ 
\vdots&\vdots&\vdots&\vdots&\vdots&\ddots\\ 
\end{array}
\end{displaymath}
The entry $c_{i,j}$ in the $i$-th row and the
$j$-th column equals
\begin{displaymath}
c_{i,j}=\frac{(i+j)!(i-j+1)}{j!(i+1)!}. 
\end{displaymath}
From the definition of $\mathbf{C}_n$ it follows that
the first two rows of $\mathbf{C}_n$ coincide.
These coinciding rows have the following interpretation
in terms of the Catalan triangle:

\begin{proposition}\label{prop11}
For $n\geq 2$ and any $j\in\{1,2,\dots,n-1\}$
we have $\mathbf{c}_{1,j}^{(n)}=c_{n-2,n-1-j}$.
\end{proposition}

\begin{proof}
For $j=n-1$ the claim follows from the definitions.
Hence it is enough to show that the entries of
the first row of $\mathbf{C}_n$ satisfy
(with respect to the first row of $\mathbf{C}_{n-1}$)
the same recursion. By Theorem~\ref{thm7}, for this
it is enough to produce a bijection between 
$\mathcal{D}_n(2,j)$ and
\begin{displaymath}
\mathcal{D}_n(1,j+1)\cup\mathcal{D}_{n-1}(1,j-1). 
\end{displaymath}

Denote by $X$ the set of all primitive Dyck paths
in $\mathcal{D}_n(2,j)$ and set $Y:=\mathcal{D}_n(2,j)\setminus X$. Every $p\in X$
has the form $rqf$ for some path 
$q\in \mathcal{D}_{n-1}(1,j-1)$ and vice versa,
for every $q\in \mathcal{D}_{n-1}(1,j-1)$ the path
$rqf$ belongs to $X$ giving us a bijection between
$X$ and $\mathcal{D}_{n-1}(1,j-1)$.

Every Dyck path $p\in Y$ has the form 
$r(rfrxf)\mathbf{f}\mathbf{r}yrf^j$, where the bold
valley $\mathbf{f}\mathbf{r}$ is the first return 
of $p$ to the diagonal. Define the map
$\varphi:Y\to \mathcal{D}_{n}(1,j+1)$ via
\begin{displaymath}
r(rfrxf)fryrf^j\mapsto
rfryr(rfrxf)^*f^j.
\end{displaymath}
We claim that $\varphi$ is a bijection, the inverse
of which is defined as follows: Every path 
$q\in \mathcal{D}_{n}(1,j+1)$ has the form
$rfrv\mathbf{r}\mathbf{r}wfrf^{j+1}$, where
the bold $\mathbf{r}\mathbf{r}$ is the last 
crossing of height $j$. The map 
$\varphi^{-1}:Y\to \mathcal{D}_{n}(1,j+1)$ is then
defined via:
\begin{displaymath}
rfrvrrwfrf^{j+1}\mapsto r(rwfrf)^*frvrf^j.
\end{displaymath}
This bijection is illustrated in Figure~\ref{fig7}.
The claim follows.
\end{proof}

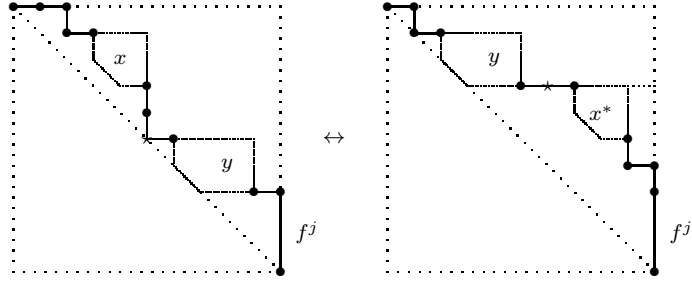
\begin{figure}
\begin{picture}(270.00,130.00)
\dottedline{4}(10.00,10.00)(110.00,10.00)
\dottedline{4}(110.00,110.00)(110.00,10.00)
\dottedline{4}(110.00,110.00)(10.00,110.00)
\dottedline{4}(10.00,10.00)(10.00,110.00)
\dottedline{4}(10.00,110.00)(110.00,10.00)
\dottedline{4}(150.00,10.00)(150.00,110.00)
\dottedline{4}(250.00,110.00)(150.00,110.00)
\dottedline{4}(250.00,110.00)(250.00,10.00)
\dottedline{4}(150.00,10.00)(250.00,10.00)
\dottedline{4}(150.00,110.00)(250.00,10.00)
%%%%%%%%%%%%%%%%%%%%%%%%%%%%%%%%%%
\put(130.00,60.00){\makebox(0,0)[cc]{\tiny $\leftrightarrow$}}
\put(50.00,90.00){\makebox(0,0)[cc]{\tiny $x$}}
\put(230.00,70.00){\makebox(0,0)[cc]{\tiny $x^*$}}
\put(90.00,50.00){\makebox(0,0)[cc]{\tiny $y$}}
\put(190.00,90.00){\makebox(0,0)[cc]{\tiny $y$}}
\put(120.00,25.00){\makebox(0,0)[cc]{\tiny $f^j$}}
\put(260.00,25.00){\makebox(0,0)[cc]{\tiny $f^j$}}
\put(10.00,110.00){\makebox(0,0)[cc]{\tiny $\bullet$}}
\put(20.00,110.00){\makebox(0,0)[cc]{\tiny $\bullet$}}
\put(30.00,110.00){\makebox(0,0)[cc]{\tiny $\bullet$}}
\put(30.00,100.00){\makebox(0,0)[cc]{\tiny $\bullet$}}
\put(40.00,100.00){\makebox(0,0)[cc]{\tiny $\bullet$}}
\put(60.00,80.00){\makebox(0,0)[cc]{\tiny $\bullet$}}
\put(60.00,70.00){\makebox(0,0)[cc]{\tiny $\bullet$}}
\put(60.00,60.00){\makebox(0,0)[cc]{\tiny $\star$}}
\put(70.00,60.00){\makebox(0,0)[cc]{\tiny $\bullet$}}
\put(100.00,40.00){\makebox(0,0)[cc]{\tiny $\bullet$}}
\put(110.00,40.00){\makebox(0,0)[cc]{\tiny $\bullet$}}
\put(110.00,10.00){\makebox(0,0)[cc]{\tiny $\bullet$}}
\put(150.00,110.00){\makebox(0,0)[cc]{\tiny $\bullet$}}
\put(160.00,110.00){\makebox(0,0)[cc]{\tiny $\bullet$}}
\put(160.00,100.00){\makebox(0,0)[cc]{\tiny $\bullet$}}
\put(170.00,100.00){\makebox(0,0)[cc]{\tiny $\bullet$}}
\put(200.00,80.00){\makebox(0,0)[cc]{\tiny $\bullet$}}
\put(210.00,80.00){\makebox(0,0)[cc]{\tiny $\star$}}
\put(220.00,80.00){\makebox(0,0)[cc]{\tiny $\bullet$}}
\put(240.00,60.00){\makebox(0,0)[cc]{\tiny $\bullet$}}
\put(240.00,50.00){\makebox(0,0)[cc]{\tiny $\bullet$}}
\put(250.00,50.00){\makebox(0,0)[cc]{\tiny $\bullet$}}
\put(250.00,40.00){\makebox(0,0)[cc]{\tiny $\bullet$}}
\put(250.00,10.00){\makebox(0,0)[cc]{\tiny $\bullet$}}
%%%%%%%%%%%%%%%%%%%%%%%%%%%%%%%%%%
\drawline(10.00,110.00)(30.00,110.00)
\drawline(30.00,100.00)(30.00,110.00)
\drawline(30.00,100.00)(40.00,100.00)
\drawline(60.00,80.00)(60.00,60.00)
\drawline(70.00,60.00)(60.00,60.00)
\drawline(100.00,40.00)(110.00,40.00)
\drawline(110.00,10.00)(110.00,40.00)
\drawline(150.00,110.00)(160.00,110.00)
\drawline(160.00,100.00)(160.00,110.00)
\drawline(160.00,100.00)(170.00,100.00)
\drawline(200.00,80.00)(220.00,80.00)
\drawline(240.00,60.00)(240.00,50.00)
\drawline(250.00,50.00)(240.00,50.00)
\drawline(250.00,50.00)(250.00,10.00)
%%%%%%%%%%%%%%%%%%%%%%%%%%%%%%%
\dottedline{1}(40.00,100.00)(60.00,100.00)
\dottedline{1}(60.00,80.00)(60.00,100.00)
\dottedline{1}(60.00,80.00)(50.00,80.00)
\dottedline{1}(40.00,90.00)(50.00,80.00)
\dottedline{1}(40.00,90.00)(40.00,100.00)
\dottedline{1}(70.00,60.00)(100.00,60.00)
\dottedline{1}(100.00,40.00)(100.00,60.00)
\dottedline{1}(100.00,40.00)(80.00,40.00)
\dottedline{1}(70.00,50.00)(80.00,40.00)
\dottedline{1}(70.00,50.00)(70.00,60.00)
\dottedline{1}(170.00,100.00)(200.00,100.00)
\dottedline{1}(200.00,80.00)(200.00,100.00)
\dottedline{1}(200.00,80.00)(180.00,80.00)
\dottedline{1}(170.00,90.00)(180.00,80.00)
\dottedline{1}(170.00,90.00)(170.00,100.00)
\dottedline{1}(220.00,80.00)(240.00,80.00)
\dottedline{1}(240.00,60.00)(240.00,80.00)
\dottedline{1}(240.00,60.00)(230.00,60.00)
\dottedline{1}(220.00,70.00)(230.00,60.00)
\dottedline{1}(220.00,70.00)(220.00,80.00)
\dottedline{2}(210.00,80.00)(250.00,80.00)
\end{picture}
\caption{Second bijection
in the proof of Proposition~\ref{prop11}, for
compactness both Dyck paths are rotated
clockwise by $\frac{\pi}{4}$, key points used
in the proof are marked with $\star$}\label{fig7} 
\end{figure}

Our next observation is related to the combinatorics
of Pascal triangle.

\begin{proposition}\label{prop17}
Let $n\in\mathbb{N}$ and $i,j\in\{1,2,\dots,n\}$. Then
we have:
\begin{enumerate}[$($a$)$] 
\item\label{prop17.1}
$\displaystyle
\mathbf{c}^{(n)}_{i,j}+
\sum_{s=1}^n\mathbf{c}^{(n)}_{s,i+j+1}=
\sum_{s=i}^n\mathbf{c}^{(n)}_{s,j+1}$;
\item\label{prop17.2}
$\displaystyle
\mathbf{c}^{(n)}_{i,j}+\mathbf{c}^{(n)}_{1,i+j}=
\mathbf{c}^{(n)}_{i+1,j}+\mathbf{c}^{(n)}_{i,j+1}$.
\end{enumerate}
\end{proposition}

\begin{proof}
Claim \eqref{prop17.2} reduces to claim \eqref{prop17.1}
using the recursion from the construction of 
$\mathbf{C}_n$ given by \eqref{eq5}.
 
By Theorem~\ref{thm7}, to prove claim \eqref{prop17.1}
we have to construct a bijection between the set
\begin{displaymath}
A:=\mathcal{D}_n(i,j)\cup
\bigcup_{s=1}^n \mathcal{D}_n(s,i+j+1)
\end{displaymath}
and the set 
\begin{displaymath}
B:=\bigcup_{s=i}^n \mathcal{D}_n(s,j+1).
\end{displaymath}
We denote by $A'$ the set $\bigcup_{s=1}^n \mathcal{D}_n(s,i+j+1)$. Let $X$ denote the set
of primitive Dyck paths in $B$ and set $Y:=B\setminus X$.

First we construct a bijection between $\mathcal{D}_n(i,j)$
and $X$.
Every $p\in \mathcal{D}_n(i,j)$ can be written in
the form $r^ifxrf^j$. Transforming this path into
$r^ixrf^{j+1}$ produces a path from $X$. This defines
a bijection with the inverse defined as follows:
Every $q\in X$ can be written in
the form $r^iyrf^{j+1}$. As $q$ is primitive,
transforming this path into $r^ifyrf^{j}$ produces
a path $\mathcal{D}_n(i,j)$ which gives the desired inverse
(see Figure~\ref{fig6}).

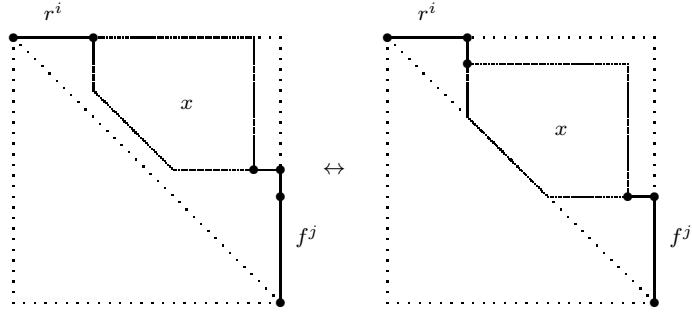
\begin{figure}
\begin{picture}(270.00,130.00)
\dottedline{4}(10.00,10.00)(110.00,10.00)
\dottedline{4}(110.00,110.00)(110.00,10.00)
\dottedline{4}(110.00,110.00)(10.00,110.00)
\dottedline{4}(10.00,10.00)(10.00,110.00)
\dottedline{4}(10.00,110.00)(110.00,10.00)
\dottedline{4}(150.00,10.00)(150.00,110.00)
\dottedline{4}(250.00,110.00)(150.00,110.00)
\dottedline{4}(250.00,110.00)(250.00,10.00)
\dottedline{4}(150.00,10.00)(250.00,10.00)
\dottedline{4}(150.00,110.00)(250.00,10.00)
%%%%%%%%%%%%%%%%%%%%%%%%%%%%%%%%%%
\put(25.00,120.00){\makebox(0,0)[cc]{\tiny $r^i$}}
\put(165.00,120.00){\makebox(0,0)[cc]{\tiny $r^i$}}
\put(120.00,35.00){\makebox(0,0)[cc]{\tiny $f^j$}}
\put(260.00,35.00){\makebox(0,0)[cc]{\tiny $f^j$}}
\put(75.00,85.00){\makebox(0,0)[cc]{\tiny $x$}}
\put(215.00,75.00){\makebox(0,0)[cc]{\tiny $x$}}
\put(130.00,60.00){\makebox(0,0)[cc]{\tiny $\leftrightarrow$}}
\put(10.00,110.00){\makebox(0,0)[cc]{\tiny $\bullet$}}
\put(40.00,110.00){\makebox(0,0)[cc]{\tiny $\bullet$}}
\put(180.00,110.00){\makebox(0,0)[cc]{\tiny $\bullet$}}
\put(180.00,100.00){\makebox(0,0)[cc]{\tiny $\bullet$}}
\put(100.00,60.00){\makebox(0,0)[cc]{\tiny $\bullet$}}
\put(110.00,60.00){\makebox(0,0)[cc]{\tiny $\bullet$}}
\put(110.00,50.00){\makebox(0,0)[cc]{\tiny $\bullet$}}
\put(240.00,50.00){\makebox(0,0)[cc]{\tiny $\bullet$}}
\put(250.00,50.00){\makebox(0,0)[cc]{\tiny $\bullet$}}
\put(250.00,10.00){\makebox(0,0)[cc]{\tiny $\bullet$}}
\put(110.00,10.00){\makebox(0,0)[cc]{\tiny $\bullet$}}
\put(150.00,110.00){\makebox(0,0)[cc]{\tiny $\bullet$}}
%%%%%%%%%%%%%%%%%%%%%%%%%%%%%%%%%%
\drawline(10.00,110.00)(40.00,110.00)
\drawline(100.00,60.00)(110.00,60.00)
\drawline(110.00,10.00)(110.00,60.00)
\drawline(150.00,110.00)(180.00,110.00)
\drawline(180.00,100.00)(180.00,110.00)
\drawline(240.00,50.00)(250.00,50.00)
\drawline(250.00,10.00)(250.00,50.00)
%%%%%%%%%%%%%%%%%%%%%%%%%%%%%%%
\dottedline{1}(40.00,110.00)(100.00,110.00)
\dottedline{1}(100.00,60.00)(100.00,110.00)
\dottedline{1}(100.00,60.00)(70.00,60.00)
\dottedline{1}(40.00,90.00)(70.00,60.00)
\dottedline{1}(40.00,90.00)(40.00,110.00)
\dottedline{1}(180.00,100.00)(240.00,100.00)
\dottedline{1}(240.00,50.00)(240.00,100.00)
\dottedline{1}(240.00,50.00)(210.00,50.00)
\dottedline{1}(180.00,80.00)(210.00,50.00)
\dottedline{1}(180.00,80.00)(180.00,100.00)
\end{picture}
\caption{First bijection
in the proof of Proposition~\ref{prop17}, for
compactness both Dyck paths are rotated
clockwise by $\frac{\pi}{4}$}\label{fig6} 
\end{figure}

It is left to construct a bijection between $A'$
and $Y$. Every $p\in Y$ can be written in the form
$r^ix\mathbf{f}\mathbf{r}yrf^{j+1}$, where the
bold $\mathbf{f}\mathbf{r}$ marks the first return of
the imprimitive path $p$ to the $x$-axis. It is easy 
to show that transforming 
\begin{displaymath}
r^ixfryrf^{j+1}\mapsto ryrx^*f^*(r^i)^*f^{j+1} 
\end{displaymath}
defines a map from $Y$ to $A'$. The inverse of this map
is defined as follows: every path $q\in A'$ can be
written in the form $ru\mathbf{r}\mathbf{r}wrf^{i+j+1}$,
where the bold $\mathbf{r}\mathbf{r}$ marks the last
crossing of the height $j+1$. It is easy 
to show that transforming 
\begin{displaymath}
rurrwrf^{i+j+1}\mapsto (f^i)^*w^*r^*rurf^{j+1}
\end{displaymath}
defines the desired inverse from $A'$ to $Y$
(see Figure~\ref{fig5}),  completing the proof.
\end{proof}

\begin{figure}
\begin{picture}(270.00,130.00)
\dottedline{4}(10.00,10.00)(110.00,10.00)
\dottedline{4}(110.00,110.00)(110.00,10.00)
\dottedline{4}(110.00,110.00)(10.00,110.00)
\dottedline{4}(10.00,10.00)(10.00,110.00)
\dottedline{4}(10.00,110.00)(110.00,10.00)
\dottedline{4}(150.00,10.00)(150.00,110.00)
\dottedline{4}(250.00,110.00)(150.00,110.00)
\dottedline{4}(250.00,110.00)(250.00,10.00)
\dottedline{4}(150.00,10.00)(250.00,10.00)
\dottedline{4}(150.00,110.00)(250.00,10.00)
%%%%%%%%%%%%%%%%%%%%%%%%%%%%%%%%%%
\put(20.00,120.00){\makebox(0,0)[cc]{\tiny $r^i$}}
\put(50.00,70.00){\makebox(0,0)[cc]{$\star$}}
\put(210.00,90.00){\makebox(0,0)[cc]{$\star$}}
\put(40.00,100.00){\makebox(0,0)[cc]{\tiny $x$}}
\put(232.00,80.00){\makebox(0,0)[cc]{\tiny $x^*$}}
\put(180.00,100.00){\makebox(0,0)[cc]{\tiny $y$}}
\put(80.00,60.00){\makebox(0,0)[cc]{\tiny $y$}}
\put(120.00,30.00){\makebox(0,0)[cc]{\tiny $f^{j+1}$}}
\put(260.00,30.00){\makebox(0,0)[cc]{\tiny $f^{j+1}$}}
\put(260.00,60.00){\makebox(0,0)[cc]{\tiny $(r^i)^*$}}
\put(130.00,60.00){\makebox(0,0)[cc]{\tiny $\leftrightarrow$}}
\put(10.00,110.00){\makebox(0,0)[cc]{\tiny $\bullet$}}
\put(30.00,110.00){\makebox(0,0)[cc]{\tiny $\bullet$}}
\put(50.00,90.00){\makebox(0,0)[cc]{\tiny $\bullet$}}
\put(50.00,80.00){\makebox(0,0)[cc]{\tiny $\bullet$}}
\put(60.00,70.00){\makebox(0,0)[cc]{\tiny $\bullet$}}
\put(100.00,50.00){\makebox(0,0)[cc]{\tiny $\bullet$}}
\put(110.00,50.00){\makebox(0,0)[cc]{\tiny $\bullet$}}
\put(110.00,10.00){\makebox(0,0)[cc]{\tiny $\bullet$}}
\put(150.00,110.00){\makebox(0,0)[cc]{\tiny $\bullet$}}
\put(160.00,110.00){\makebox(0,0)[cc]{\tiny $\bullet$}}
\put(200.00,90.00){\makebox(0,0)[cc]{\tiny $\bullet$}}
\put(220.00,90.00){\makebox(0,0)[cc]{\tiny $\bullet$}}
\put(240.00,70.00){\makebox(0,0)[cc]{\tiny $\bullet$}}
\put(250.00,70.00){\makebox(0,0)[cc]{\tiny $\bullet$}}
\put(250.00,50.00){\makebox(0,0)[cc]{\tiny $\bullet$}}
\put(250.00,10.00){\makebox(0,0)[cc]{\tiny $\bullet$}}
%%%%%%%%%%%%%%%%%%%%%%%%%%%%%%%%%%
\drawline(10.00,110.00)(30.00,110.00)
\drawline(50.00,90.00)(50.00,70.00)
\drawline(60.00,70.00)(50.00,70.00)
\drawline(100.00,50.00)(110.00,50.00)
\drawline(110.00,10.00)(110.00,50.00)
\drawline(150.00,110.00)(160.00,110.00)
\drawline(200.00,90.00)(220.00,90.00)
\drawline(240.00,70.00)(250.00,70.00)
\drawline(250.00,10.00)(250.00,70.00)
%%%%%%%%%%%%%%%%%%%%%%%%%%%%%%%
\dottedline{1}(30.00,110.00)(50.00,110.00)
\dottedline{1}(50.00,80.00)(50.00,110.00)
\dottedline{1}(50.00,80.00)(30.00,100.00)
\dottedline{1}(30.00,110.00)(30.00,100.00)
\dottedline{1}(60.00,70.00)(100.00,70.00)
\dottedline{1}(100.00,50.00)(100.00,70.00)
\dottedline{1}(100.00,50.00)(70.00,50.00)
\dottedline{1}(60.00,60.00)(70.00,50.00)
\dottedline{1}(60.00,60.00)(60.00,70.00)
\dottedline{1}(210.00,90.00)(240.00,90.00)
\dottedline{1}(240.00,70.00)(240.00,90.00)
\dottedline{1}(240.00,70.00)(230.00,70.00)
\dottedline{1}(210.00,90.00)(230.00,70.00)
\dottedline{1}(160.00,110.00)(200.00,110.00)
\dottedline{1}(200.00,90.00)(200.00,110.00)
\dottedline{1}(200.00,90.00)(170.00,90.00)
\dottedline{1}(160.00,100.00)(170.00,90.00)
\dottedline{1}(160.00,100.00)(160.00,110.00)
\dottedline{2}(230.00,70.00)(250.00,50.00)
\end{picture}
\caption{Second bijection
in the proof of Proposition~\ref{prop17}, for
compactness both Dyck paths are rotated
clockwise by $\frac{\pi}{4}$, key points used
in the proof are marked with $\star$}\label{fig5} 
\end{figure}

Denote by $\mathbf{C}'_n$ the principal minor of
$\mathbf{C}_n$ corresponding to the first $n-1$
rows and columns. Proposition~\ref{prop17} implies
that the part of $\mathbf{C}'_n$ lying below and on
the opposite diagonal coincides with the Pascal
triangle (appropriately rotated). The part above
the opposite diagonal follows almost the same
recursion (every element is the sum of the 
right and below neighbors) but with a 
kind of a ``correction term''. 

Here is an explicit formula for 
$\mathbf{c}^{(n)}_{i,j}$:

\begin{proposition}\label{prop23}
For $n\in\mathbb{N}$ and $i,j\in\{1,2,\dots,n-1\}$ we have:
\begin{displaymath}
\mathbf{c}^{(n)}_{i,j}= 
\binom{(n-1-i)+(n-1-j)}{n-1-i}
-\binom{(n-1-i)+(n-1-j)}{n-i-j-1}.
\end{displaymath}
\end{proposition}

\begin{proof}
We have to count the number of paths in the plane
from the point $(i+1,i-1)$ to the point $(2n-j-1,j-1)$ along
$(1,1)$ and $(1,-1)$ which do not go below the
$x$-axis. This is a slight generalization of the
classical ballot problem (see \cite{An}), which 
can be solved using Andr{\'e}'s reflection principle
(see e.g. \cite[Section~1.6]{Gr} for application
of this principle  to Catalan combinatorics). 
The number of all paths from $(i+1,i-1)$ to 
$(2n-j-1,j-1)$ along $(1,1)$ and $(1,-1)$ equals
$\binom{(n-1-i)+(n-1-j)}{n-1-i}$ (the first summand
of our formula), since we have to do 
$(n-1-i)+(n-1-j)$ steps, and arbitrary $n-1-i$ of them
can be chosen to be rises.

Let $p$ be a path from $(i+1,i-1)$ to 
$(2n-j-1,j-1)$  which goes below the $x$-axis.
Then $p$ has the form $xffy$ where $ff$ indicates the
first crossing of the $x$ axis.  Denote by $y'$
the path obtained from $y$ by swapping all rises and
falls. Then $xffy'$ is a path from $(i+1,i-1)$ to
$(2n-j-1,-j-1)$. Conversely, every path from 
$(i+1,i-1)$ to $(2n-j-1,-j-1)$ crosses the $x$-axis.
Then we can write this path in the form
$uffw$ where $ff$ indicates the
first crossing of the $y$ axis. Then the path
$uffw'$, where $w'$ is obtained from $w$ by 
swapping all rises and falls, is a path from $(i+1,i-1)$ 
to  $(2n-j-1,j-1)$ which goes below the $x$-axes.

Now the claim follows from the observation that 
the number of paths from $(i+1,i-1)$ to $(2n-j-1,-j-1)$
equals $\binom{(n-1-i)+(n-1-j)}{n-i-j-1}$
(the second summand in our formula).
\end{proof}

\subsection{Proof of Theorem~\ref{tmain}}\label{s3.3}

For $n>1$ there is an obvious bijection between strong
ideals of $\hat{\mathfrak{sl}}_n$ and $\hat{\mathfrak{gl}}_n$.
Therefore it is convenient to use the convention that
$\hat{\mathfrak{sl}}_1$ is a codimension one subalgebra 
of $\hat{\mathfrak{gl}}_1$.

We denote by $\mathfrak{B}_n$ the (finite) set of basic 
ideals for the affine Lie algebra $\hat{\mathfrak{sl}}_n$
and set $\mathbf{b}_n:=|\mathfrak{B}_n|$. The aim
of this subsection is to show that $\mathbf{b}_n=
\mathbf{C}_n\cdot\omega\, \mathbf{C}_n$ as stated in
Theorem~\ref{tmain}.

With every $\mathfrak{i}\in\mathfrak{B}_n$ we associate
a pair $\Phi(\mathfrak{i}):=(p,q)$ of 
Dyck paths of semilength $n$
in the following way:
View an $n\times n$ matrix as a square in the coordinate
plane with vertexes $(0,0)$, $(n,0)$, $(0,-n)$
and $(n,-n)$ in the natural way (the entries of the
matrix are $1\times 1$ boxes of this square). 
Our convention is
that the simple root $\alpha_i$, $i=1,\dots,n-1$, of
$\Delta_+$ corresponds to the box with the north-west
coordinate $(i,1-i)$ and the opposite root $-\alpha_i$
corresponds to the box with the north-west 
coordinate $(i-1,-i)$. Fill in all boxes corresponding to
the roots in $\mathbf{s}_+(\mathfrak{i}):=\Delta_+\cap \mathrm{supp}(\mathfrak{i})$.
They all belong to the part of our matrix above the 
diagonal. Since
$\mathfrak{i}$ is an ideal, if some box is filled, then
both its right and upper neighbors are filled as well
(provided that they belong to our square). Consider the
line $l$ starting at $(0,0)$ and ending at $(n,-n)$
which goes along the vectors $(1,0)$ and $(0,-1)$
separating the filled boxed from the ones which are
not filled. Rotate this line counterclockwise around
the origin by
$\frac{\pi}{4}$ and scale it appropriately to obtain
the Dyck path $p$ (of semilength $n$). 
An example is shown in Figure~\ref{fig2}.

The construction of $q$ is similar. Fill in all 
boxes corresponding to the roots $\alpha\in\Delta_-$ 
such that  $\alpha+\delta\in \mathrm{supp}(\mathfrak{i})$
(we denote the set of all such $\alpha$ by $\mathbf{s}_-(\mathfrak{i})$).
They all belong to the part of our matrix below the 
diagonal. Since
$\mathfrak{i}$ is an ideal, if some box is filled, then
both its right and upper neighbors are filled as well
(provided that they belong to the part of our square
below the diagonal). Consider the
line $l$ starting at $(0,0)$ and ending at $(n,-n)$
which goes along the vectors $(1,0)$ and $(0,-1)$
separating the filled boxed below the diagonal from the ones
which are not filled. Rotate this line counterclockwise by
$\frac{\pi}{4}$, reflect it in the $x$-axis and 
scale it appropriately to obtain the Dyck path $q$
(or semilength $n$).
An example is shown in Figure~\ref{fig3}.

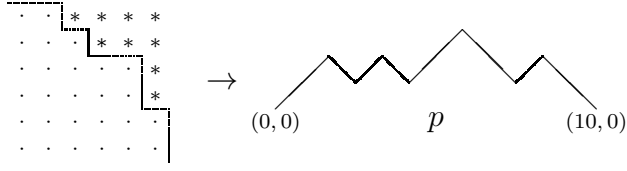
\begin{figure}
\begin{picture}(240.00,80.00)
\put(15.00,15.00){\makebox(0,0)[cc]{\tiny $\cdot$}}
\put(25.00,15.00){\makebox(0,0)[cc]{\tiny $\cdot$}}
\put(35.00,15.00){\makebox(0,0)[cc]{\tiny $\cdot$}}
\put(45.00,15.00){\makebox(0,0)[cc]{\tiny $\cdot$}}
\put(55.00,15.00){\makebox(0,0)[cc]{\tiny $\cdot$}}
\put(65.00,15.00){\makebox(0,0)[cc]{\tiny $\cdot$}}
\put(15.00,25.00){\makebox(0,0)[cc]{\tiny $\cdot$}}
\put(25.00,25.00){\makebox(0,0)[cc]{\tiny $\cdot$}}
\put(35.00,25.00){\makebox(0,0)[cc]{\tiny $\cdot$}}
\put(45.00,25.00){\makebox(0,0)[cc]{\tiny $\cdot$}}
\put(55.00,25.00){\makebox(0,0)[cc]{\tiny $\cdot$}}
\put(65.00,25.00){\makebox(0,0)[cc]{\tiny $\cdot$}}
\put(15.00,35.00){\makebox(0,0)[cc]{\tiny $\cdot$}}
\put(25.00,35.00){\makebox(0,0)[cc]{\tiny $\cdot$}}
\put(35.00,35.00){\makebox(0,0)[cc]{\tiny $\cdot$}}
\put(45.00,35.00){\makebox(0,0)[cc]{\tiny $\cdot$}}
\put(55.00,35.00){\makebox(0,0)[cc]{\tiny $\cdot$}}
\put(55.00,45.00){\makebox(0,0)[cc]{\tiny $\cdot$}}
\put(15.00,45.00){\makebox(0,0)[cc]{\tiny $\cdot$}}
\put(25.00,45.00){\makebox(0,0)[cc]{\tiny $\cdot$}}
\put(35.00,45.00){\makebox(0,0)[cc]{\tiny $\cdot$}}
\put(45.00,45.00){\makebox(0,0)[cc]{\tiny $\cdot$}}
\put(15.00,55.00){\makebox(0,0)[cc]{\tiny $\cdot$}}
\put(25.00,55.00){\makebox(0,0)[cc]{\tiny $\cdot$}}
\put(35.00,55.00){\makebox(0,0)[cc]{\tiny $\cdot$}}
\put(45.00,55.00){\makebox(0,0)[cc]{\tiny $\cdot$}}
\put(15.00,65.00){\makebox(0,0)[cc]{\tiny $\cdot$}}
\put(25.00,65.00){\makebox(0,0)[cc]{\tiny $\cdot$}}
\put(35.00,65.00){\makebox(0,0)[cc]{\tiny $\ast$}}
\put(45.00,65.00){\makebox(0,0)[cc]{\tiny $\ast$}}
\put(55.00,65.00){\makebox(0,0)[cc]{\tiny $\ast$}}
\put(65.00,65.00){\makebox(0,0)[cc]{\tiny $\ast$}}
\put(55.00,55.00){\makebox(0,0)[cc]{\tiny $\ast$}}
\put(65.00,55.00){\makebox(0,0)[cc]{\tiny $\ast$}}
\put(45.00,55.00){\makebox(0,0)[cc]{\tiny $\ast$}}
\put(65.00,45.00){\makebox(0,0)[cc]{\tiny $\ast$}}
\put(65.00,35.00){\makebox(0,0)[cc]{\tiny $\ast$}}
%%%%%%%%%%%%%%%%%%%%%%%%%%%%%%%%%%%%%%%%%%%%%%
\dottedline{1}(10.00,70.00)(30.00,70.00)
\dottedline{1}(30.00,60.00)(30.00,70.00)
\dottedline{1}(30.00,60.00)(40.00,60.00)
\dottedline{1}(40.00,50.00)(40.00,60.00)
\dottedline{1}(40.00,50.00)(60.00,50.00)
\dottedline{1}(60.00,30.00)(60.00,50.00)
\dottedline{1}(60.00,30.00)(70.00,30.00)
\dottedline{1}(70.00,10.00)(70.00,30.00)
%%%%%%%%%%%%%%%%%%%%%%%%%%%%%%%%%%%%%%%%%%%%%%%%%
\put(90.00,40.00){\makebox(0,0)[cc]{$\to$}}
\put(110.00,25.00){\makebox(0,0)[cc]{\tiny $(0,0)$}}
\put(230.00,25.00){\makebox(0,0)[cc]{\tiny $(10,0)$}}
\put(170.00,25.00){\makebox(0,0)[cc]{$p$}}
%%%%%%%%%%%%%%%%%%%%%%%%%%%%%%%%%%%%%%%%%%%%%%
\drawline(110.00,30.00)(130.00,50.00)
\drawline(140.00,40.00)(130.00,50.00)
\drawline(140.00,40.00)(150.00,50.00)
\drawline(160.00,40.00)(150.00,50.00)
\drawline(160.00,40.00)(180.00,60.00)
\drawline(200.00,40.00)(180.00,60.00)
\drawline(200.00,40.00)(210.00,50.00)
\drawline(230.00,30.00)(210.00,50.00)
%%%%%%%%%%%%%%%%%%%%%%%%%%%%%%%%%%%%%%%%%%%%%%%
\end{picture}
\caption{The Dyck path $p$ associated to
$\mathfrak{i}$, the elements of 
$\mathbf{s}_+(\mathfrak{i})$ are given by $\ast$}\label{fig2} 
\end{figure}

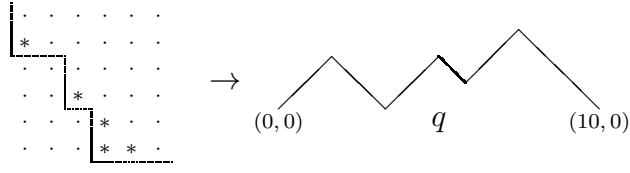
\begin{figure}
\begin{picture}(240.00,80.00)
\put(15.00,15.00){\makebox(0,0)[cc]{\tiny $\cdot$}}
\put(25.00,15.00){\makebox(0,0)[cc]{\tiny $\cdot$}}
\put(35.00,15.00){\makebox(0,0)[cc]{\tiny $\cdot$}}
\put(45.00,15.00){\makebox(0,0)[cc]{\tiny $\ast$}}
\put(55.00,15.00){\makebox(0,0)[cc]{\tiny $\ast$}}
\put(65.00,15.00){\makebox(0,0)[cc]{\tiny $\cdot$}}
\put(15.00,25.00){\makebox(0,0)[cc]{\tiny $\cdot$}}
\put(25.00,25.00){\makebox(0,0)[cc]{\tiny $\cdot$}}
\put(35.00,25.00){\makebox(0,0)[cc]{\tiny $\cdot$}}
\put(45.00,25.00){\makebox(0,0)[cc]{\tiny $\ast$}}
\put(55.00,25.00){\makebox(0,0)[cc]{\tiny $\cdot$}}
\put(65.00,25.00){\makebox(0,0)[cc]{\tiny $\cdot$}}
\put(15.00,35.00){\makebox(0,0)[cc]{\tiny $\cdot$}}
\put(25.00,35.00){\makebox(0,0)[cc]{\tiny $\cdot$}}
\put(35.00,35.00){\makebox(0,0)[cc]{\tiny $\ast$}}
\put(45.00,35.00){\makebox(0,0)[cc]{\tiny $\cdot$}}
\put(65.00,35.00){\makebox(0,0)[cc]{\tiny $\cdot$}}
\put(55.00,35.00){\makebox(0,0)[cc]{\tiny $\cdot$}}
\put(15.00,45.00){\makebox(0,0)[cc]{\tiny $\cdot$}}
\put(25.00,45.00){\makebox(0,0)[cc]{\tiny $\cdot$}}
\put(35.00,45.00){\makebox(0,0)[cc]{\tiny $\cdot$}}
\put(45.00,45.00){\makebox(0,0)[cc]{\tiny $\cdot$}}
\put(65.00,45.00){\makebox(0,0)[cc]{\tiny $\cdot$}}
\put(55.00,45.00){\makebox(0,0)[cc]{\tiny $\cdot$}}
\put(15.00,55.00){\makebox(0,0)[cc]{\tiny $\ast$}}
\put(25.00,55.00){\makebox(0,0)[cc]{\tiny $\cdot$}}
\put(35.00,55.00){\makebox(0,0)[cc]{\tiny $\cdot$}}
\put(45.00,55.00){\makebox(0,0)[cc]{\tiny $\cdot$}}
\put(65.00,55.00){\makebox(0,0)[cc]{\tiny $\cdot$}}
\put(55.00,55.00){\makebox(0,0)[cc]{\tiny $\cdot$}}
\put(15.00,65.00){\makebox(0,0)[cc]{\tiny $\cdot$}}
\put(25.00,65.00){\makebox(0,0)[cc]{\tiny $\cdot$}}
\put(35.00,65.00){\makebox(0,0)[cc]{\tiny $\cdot$}}
\put(45.00,65.00){\makebox(0,0)[cc]{\tiny $\cdot$}}
\put(65.00,65.00){\makebox(0,0)[cc]{\tiny $\cdot$}}
\put(55.00,65.00){\makebox(0,0)[cc]{\tiny $\cdot$}}
%%%%%%%%%%%%%%%%%%%%%%%%%%%%%%%%%%%%%%%%%%%%%%
\dottedline{1}(10.00,70.00)(10.00,50.00)
\dottedline{1}(30.00,50.00)(10.00,50.00)
\dottedline{1}(30.00,50.00)(30.00,30.00)
\dottedline{1}(40.00,30.00)(30.00,30.00)
\dottedline{1}(40.00,30.00)(40.00,10.00)
\dottedline{1}(70.00,10.00)(40.00,10.00)
%%%%%%%%%%%%%%%%%%%%%%%%%%%%%%%%%%%%%%%%%%%%%%%%%
\put(90.00,40.00){\makebox(0,0)[cc]{$\to$}}
\put(110.00,25.00){\makebox(0,0)[cc]{\tiny $(0,0)$}}
\put(230.00,25.00){\makebox(0,0)[cc]{\tiny $(10,0)$}}
\put(170.00,25.00){\makebox(0,0)[cc]{$q$}}
%%%%%%%%%%%%%%%%%%%%%%%%%%%%%%%%%%%%%%%%%%%%%%
\drawline(110.00,30.00)(130.00,50.00)
\drawline(150.00,30.00)(130.00,50.00)
\drawline(150.00,30.00)(170.00,50.00)
\drawline(180.00,40.00)(170.00,50.00)
\drawline(180.00,40.00)(200.00,60.00)
\drawline(230.00,30.00)(200.00,60.00)
%%%%%%%%%%%%%%%%%%%%%%%%%%%%%%%%%%%%%%%%%%%%%%%
\end{picture}
\caption{The Dyck path $q$ associated to
$\mathfrak{j}$, the elements of 
$\mathbf{s}_-(\mathfrak{j})$ are given by $\ast$}\label{fig3} 
\end{figure}

From Theorem~\ref{thm2}\eqref{thm2.4} it follows that
different ideals in $\mathfrak{B}_n$ give different 
pairs of Dyck paths. However, as we will see below,
not every pair of Dyck paths can be obtained in such
way. A pair of Dyck paths which corresponds in this way to
some ideal in $\mathfrak{B}_n$ will be called
{\em admissible}. We will prove Theorem~\ref{tmain}
by counting the number of admissible pairs of Dyck
paths. Denote by $\mathfrak{P}_n$ the set of all
admissible pairs of Dyck paths and for $k,m\in
\{1,2,\dots,n\}$ let $\mathfrak{P}_{n}(k,m)$ denote the
set of all pairs $(p,q)$ from $\mathfrak{P}_n$ for which $k$
and $m$ are the heights of the first and the last peak
of $p$, respectively. Then $\mathfrak{P}_n$ is the
disjoint union of the $\mathfrak{P}_n(k,m)$'s.

\begin{lemma}\label{lem9}
We have $(p,q)\in \mathfrak{P}_n(k,m)$ if and only if
the first peak of $q$ has height at least 
$n-m$ and the last peak of $q$ has height at least 
$n-k$.
\end{lemma}

\begin{proof}
Let $\mathfrak{i}$ be the basic ideal corresponding to
$(p,q)$. Then, by Proposition~\ref{prop4}, for
any $\alpha\in\Delta_+\cap\mathrm{supp}(\mathfrak{i})$
every element in $\gamma\in D$ such that 
$\gamma\geq \alpha$ must belong to 
$\mathrm{supp}(\mathfrak{i})$. This means that for every 
$\beta\in\Delta_-$ with the property $\beta\geq
\alpha-\alpha_{\mathrm{max}}$, for some
$\alpha\in\Delta_+\cap\mathrm{supp}(\mathfrak{i})$,
the root $\beta+\delta$ belongs to 
$\mathrm{supp}(\mathfrak{i})$.

If $\alpha=\alpha_{s}+\alpha_{s+1}+\dots+\alpha_r$
for some $s,r$ such that $1\leq s\leq r\leq n-1$,
then every $\beta\in\Delta_-$ with the property $\beta\geq
\alpha-\alpha_{\mathrm{max}}$ has the form
$-(\alpha_{i}+\alpha_{i+1}+\dots+\alpha_j)$ either
for some $j<s$ or some $i>r$. The maximal value of
$s$ is achieved by the root which corresponds to
the box of our square for which the south-west corner
is the last valley of $p$.  The last peak of $p$
has height $n-s=m$. This implies that
the first peak of $q$ has height at least 
$n-m$, see Figure~\ref{fig4}.

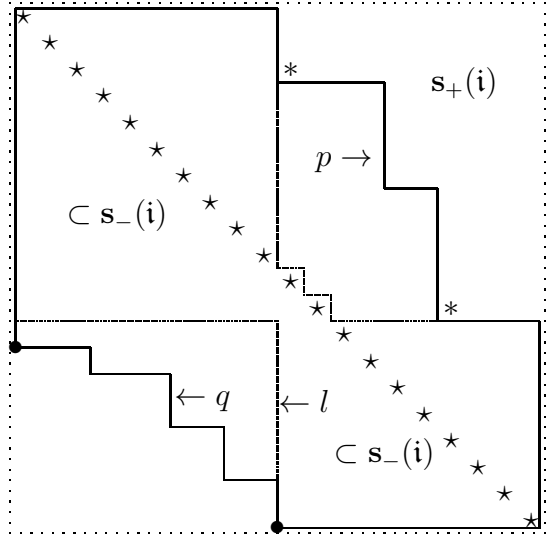
\begin{figure}
\begin{picture}(220.00,220.00)
\dottedline{4}(10.00,10.00)(10.00,210.00)
\dottedline{4}(210.00,210.00)(10.00,210.00)
\dottedline{4}(210.00,210.00)(210.00,10.00)
\dottedline{4}(10.00,10.00)(210.00,10.00)
%%%%%%%%%%%%%%%%%%%%%%%%%%%%%%%%%%
\put(15.00,205.00){\makebox(0,0)[cc]{$\star$}}
\put(25.00,195.00){\makebox(0,0)[cc]{$\star$}}
\put(35.00,185.00){\makebox(0,0)[cc]{$\star$}}
\put(45.00,175.00){\makebox(0,0)[cc]{$\star$}}
\put(55.00,165.00){\makebox(0,0)[cc]{$\star$}}
\put(65.00,155.00){\makebox(0,0)[cc]{$\star$}}
\put(75.00,145.00){\makebox(0,0)[cc]{$\star$}}
\put(85.00,135.00){\makebox(0,0)[cc]{$\star$}}
\put(95.00,125.00){\makebox(0,0)[cc]{$\star$}}
\put(105.00,115.00){\makebox(0,0)[cc]{$\star$}}
\put(115.00,105.00){\makebox(0,0)[cc]{$\star$}}
\put(125.00,95.00){\makebox(0,0)[cc]{$\star$}}
\put(135.00,85.00){\makebox(0,0)[cc]{$\star$}}
\put(145.00,75.00){\makebox(0,0)[cc]{$\star$}}
\put(155.00,65.00){\makebox(0,0)[cc]{$\star$}}
\put(165.00,55.00){\makebox(0,0)[cc]{$\star$}}
\put(175.00,45.00){\makebox(0,0)[cc]{$\star$}}
\put(185.00,35.00){\makebox(0,0)[cc]{$\star$}}
\put(195.00,25.00){\makebox(0,0)[cc]{$\star$}}
\put(205.00,15.00){\makebox(0,0)[cc]{$\star$}}
%%%%%%%%%%%%%%%%%%%%%%%%%%%%%%%%%%
\drawline(12.00,208.00)(110.00,208.00)
\drawline(110.00,180.00)(110.00,208.00)
\drawline(110.00,180.00)(150.00,180.00)
\drawline(150.00,140.00)(150.00,180.00)
\drawline(150.00,140.00)(170.00,140.00)
\drawline(170.00,90.00)(170.00,140.00)
\drawline(170.00,90.00)(208.00,90.00)
\drawline(208.00,12.00)(208.00,90.00)
\drawline(208.00,12.00)(110.00,12.00)
\drawline(110.00,30.00)(110.00,12.00)
\drawline(110.00,30.00)(90.00,30.00)
\drawline(90.00,50.00)(90.00,30.00)
\drawline(90.00,50.00)(70.00,50.00)
\drawline(70.00,70.00)(70.00,50.00)
\drawline(70.00,70.00)(40.00,70.00)
\drawline(40.00,80.00)(40.00,70.00)
\drawline(40.00,80.00)(12.00,80.00)
\drawline(12.00,208.00)(12.00,80.00)
%%%%%%%%%%%%%%%%%%%%%%%%%%%%%%%
\dottedline{1}(110.00,30.00)(110.00,90.00)
\dottedline{1}(12.00,90.00)(110.00,90.00)
\dottedline{1}(110.00,180.00)(110.00,110.00)
\dottedline{1}(120.00,110.00)(110.00,110.00)
\dottedline{1}(120.00,110.00)(120.00,100.00)
\dottedline{1}(130.00,100.00)(120.00,100.00)
\dottedline{1}(130.00,100.00)(130.00,90.00)
\dottedline{1}(170.00,90.00)(130.00,90.00)
%%%%%%%%%%%%%%%%%%%%%%%%%%%%%%%%%%
\put(175.00,95.00){\makebox(0,0)[cc]{$\ast$}}
\put(115.00,185.00){\makebox(0,0)[cc]{$\ast$}}
\put(180.00,180.00){\makebox(0,0)[cc]
{$\mathbf{s}_+(\mathfrak{i})$}}
\put(50.00,130.00){\makebox(0,0)[cc]
{$\subset\mathbf{s}_-(\mathfrak{i})$}}
\put(150.00,40.00){\makebox(0,0)[cc]
{$\subset\mathbf{s}_-(\mathfrak{i})$}}
\put(135.00,150.00){\makebox(0,0)[cc]{$p\to$}}
\put(82.00,60.00){\makebox(0,0)[cc]{$\leftarrow q$}}
\put(120.00,60.00){\makebox(0,0)[cc]{$\leftarrow l$}}
\put(110.00,12.00){\makebox(0,0)[cc]{$\bullet$}}
\put(12.00,80.00){\makebox(0,0)[cc]{$\bullet$}}
\end{picture}
\caption{Argument in the proof of 
Lemma~\ref{lem9}: asterisks show the first and the 
last valleys of $p$ which determine the dashed line $l$, 
the region between $l$ and the diagonal must belong to
$\mathbf{s}_-(\mathfrak{i})$ and hence the path 
$q$ must be below $l$, the first and the last peaks of $q$ are marked by $\bullet$}\label{fig4} 
\end{figure}

The minimal value of
$r$ is achieved by the root which corresponds to
the box of our square for which the south-west corner
is the first valley of $p$. The first peak of $p$
has height $r=k$. This implies that
the last peak of $q$ has height at least 
$n-k$. 
 
Conversely, let $(p,q)$ be a pair of Dyck paths.
Assume that if the first peak of $q$ has 
height at least $n-m$ and the  last peak of $q$ has 
height at least $n-k$. Consider the set $I$ of
all roots in $\Delta_+$ which correspond to the boxes
to the north-east of $p$ in the matrix setup
together will all roots of the form
$\beta+\delta$, where $\beta$ corresponds to the 
boxes to the north-east of $q$ in the matrix setup
(i.e. between $q$ and the diagonal).  Then $I\cup\{\delta\}
\subset D$ is a non-empty subset and the argument 
from above implies that $\alpha\in I$,
$\beta\in D$ and $\beta\geq \alpha$
implies $\beta\in I\cup\{\delta\}$. This means that
$(p,q)$ is admissible.
\end{proof}

\begin{proof}[Proof of Theorem~\ref{tmain}]
From the definitions we have: 
\begin{displaymath}
\mathbf{b}_n=|\mathfrak{P}_n|=
\left|\bigcup_{k,m=1}^n\mathfrak{P}_n(k,m)\right|=
\sum_{k,m=1}^n\left|\mathfrak{P}_n(k,m)\right|.
\end{displaymath}
From Lemma~\ref{lem9} it follows that there
is a bijection between $\mathfrak{P}_n(k,m)$
and the set
\begin{displaymath}
\bigcup_{i=n-m}^n\,\,\bigcup_{j=n-k}^n
\mathcal{D}_n(i,j)\times\mathcal{D}_n(k,m).
\end{displaymath}
The assertion of Theorem~\ref{tmain} follows
now from Theorem~\ref{thm7} and the definitions.
\end{proof}

\begin{remark}\label{rem10}
{\rm 
The integral sequence $\{\mathbf{b}_n:n\geq 1\}$
seems to be new (at least we could not find anything
similar in \cite{OEIS}). The values of this sequence
for small $n$ are:
\begin{displaymath}
1,\,4,\,18,\,82,\,370,\,1648,\,7252,\,31582,\,136338,\,584248\dots
\end{displaymath}
}
\end{remark}

\begin{corollary}\label{cor22}
For $n\in\mathbb{N}$ we have
\begin{displaymath}
\mathbf{b}_n=\sum_{i,j=1}^n\mathbf{c}^{(n)}_{i,j}
\sum_{k=n-i}^n\sum_{m=n-j}^n\mathbf{c}^{(n)}_{k,m}.
\end{displaymath}
\end{corollary}

Combining this with Proposition~\ref{prop23}
gives and explicit formula for $\mathbf{b}_n$.

\subsection{Algebraic properties of 
combinatorial ideals}\label{s3.4}

Let $\mathfrak{j}$ be a combinatorial ideal.
Following \cite{Pa},
a root $\alpha\in\Delta_+$ is said to be a 
{\em generator} of $\mathfrak{j}$ if
$\alpha\in \mathrm{supp}(\mathfrak{j})$ and
\begin{displaymath}
\bigoplus_{\beta\in \mathrm{supp}
(\mathfrak{j})\setminus\{\alpha\}}
\hat{\mathfrak{g}}_{\beta}
\end{displaymath}
is a combinatorial ideal. It is easy to see that the
number of generators is in fact constant on the
equivalence classes of combinatorial ideals.
Hence it is enough to consider basic ideals.

Let $\mathfrak{i}$ be a basic ideal for
$\hat{\mathfrak{sl}}_n$ and $\Phi(\mathfrak{i})=(p,q)$. 
In the classical case
(studied in \cite{Pa}) the number of generators of
an ad-nilpotent ideal in the Borel subalgebra of
$\mathfrak{sl}_n$ equals the number of valleys in the
associated Dyck path. In our case we have:

\begin{proposition}\label{prop24}
Assume that $p\in\mathcal{D}_n(a,b)$ and
$q\in\mathcal{D}_n(c,d)$. Then 
the number of generators of $\mathfrak{i}$ 
equals $1$ if $D\cap\mathrm{supp}(\mathfrak{i})=
\{\delta\}$ and
\begin{displaymath}
\mathbf{v}(p)+\mathbf{p}_1(q)-\delta_{d,n-a}
-\delta_{c,n-b}, 
\end{displaymath}
where $\delta_{i,j}$ is the Kronecker symbol, otherwise.
\end{proposition}

\begin{proof}
The claim is obvious in the case 
$D\cap\mathrm{supp}(\mathfrak{i})=\{\delta\}$.
In all other cases 
the proof is best understood by looking at
Figure~\ref{fig4}. The generators of 
$\mathbf{s}_+(\mathfrak{i})$ are boxes whose
south-west corner is a valley of $p$
(this is the result from \cite{Pa}).
The region between $l$ and the diagonal
belongs to $\mathfrak{i}$ automatically and thus
contains no generators. The remaining generators
of $\mathbf{s}_-(\mathfrak{i})$ are boxes whose
south-west corner is a peak of $q$. Hence we
have to count all peaks of $q$ with three exceptions:
we should not count the first peak of $q$ if it
coincides with the first peak of $l$ and similarly
for the last peak; we should not count peaks of
height one as the corresponding boxes are on the
diagonal and hence do not correspond to any roots
(or rather correspond to the imaginary root which
is not a generator by the proof of Theorem~\ref{thm2}). 
The claim follows. 
\end{proof}

Set
\begin{displaymath}
\mathfrak{a}=\bigoplus_{\alpha\in D'}
\hat{\mathfrak{g}}_{\alpha}.
\end{displaymath}
A basic ideal $\mathfrak{i}$ will be called 
{\em quasi-abelian} provided that the image of
$\mathfrak{i}$ in 
$\hat{\mathfrak{n}}_+/\mathfrak{a}$
is an abelian ideal. In the finite dimensional
case it is known that the number of abelian
ad-nilpotent ideals in $\mathfrak{b}$ equals
$2^{\dim \mathfrak{h}}$, see \cite{CP1}.

For $p\in\mathcal{D}_n(i,j)$ denote by 
$\overline{p}$ the unique minimal (with respect to
$\leq$) path in $\mathcal{D}_n(n-j,n-i)$
(here we set $\mathcal{D}_n(0,0):=\mathcal{D}_n(1,1)$). Note that
Lemma~\ref{lem9} can then be reformulated as follows:
$(p,q)\in\mathfrak{P}_n(k,m)$ if and only if
$\overline{p}\leq q$.

\begin{proposition}\label{prop27}
The map $\Phi$ induces a bijection between the 
set of quasi-abelian basic ideals in
$\hat{\mathfrak{b}}$ and the set of all pairs $(p,q)$
of Dyck paths satisfying the condition
$\overline{p}\leq q\leq p$.
\end{proposition}

\begin{proof}
Let $\mathfrak{i}$ be a  basic ideal
and $(p,q)=\Phi(i)$. Then $\overline{p}\leq q$.
The condition $q\leq p$ is equivalent to the following
condition: for every $\alpha\in \mathbf{s}_+(\mathfrak{i})$
we have $-\alpha\not\in \mathbf{s}_-(\mathfrak{i})$,
which is obviously necessary for the ideal $\mathfrak{i}$ 
to be quasi-abelian.

On the other hand, $\overline{p}\leq p$ implies
$(n-i)+(n-j)\leq n$. This yields that 
there exist a simple root $\alpha_m$ in $\Delta_+$ such that 
every root in $\mathbf{s}_+(\mathfrak{i})$ has the form
$\alpha_i+\alpha_{i+1}+\dots+\alpha_j$ for some 
$i\leq m\leq j$. Hence 
$\mathfrak{i}\cap\mathfrak{n}_+$ is an abelian ideal
of $\mathfrak{n}_+$. Furthermore, the condition 
\begin{displaymath}
\alpha\in \mathbf{s}_+(\mathfrak{i})\quad
\text{ implies }\quad -\alpha\not\in \mathbf{s}_-(\mathfrak{i}) 
\end{displaymath}
also means that for any 
$\alpha\in \mathbf{s}_+(\mathfrak{i})$ and
$\beta\in \mathbf{s}_-(\mathfrak{i})$ we have
$[\hat{\mathfrak{g}}_{\alpha},\hat{\mathfrak{g}}_{\beta}]
\in\mathfrak{a}$, which forces the ideal $\mathfrak{i}$
to be quasi-abelian. The claim follows.
\end{proof}

\begin{problem}\label{problem28}
{\em 
Find an explicit formula for the number of 
quasi-abelian ideals in $\hat{\mathfrak{sl}}_n$.
}
\end{problem}

For small values of $n$ the answer to
Problem~\ref{problem28} (computed using a dull examination
of all possible cases on a computer) is as follows:
\begin{displaymath}
1,\,3,\,11,\,44,\,183,\,774,\,3294,\,14034,\,59711,\,253430,\dots 
\end{displaymath}
This sequence again seems to be new and does not appear in \cite{OEIS}.

As usual, for an ideal $\mathfrak{i}$ consider the lower central series
$\mathfrak{i}^{0}:=\mathfrak{i}$ and $\mathfrak{i}^{m+1}:=[\mathfrak{i}^m,\mathfrak{i}]$,
$m\geq 0$. The ideal $\mathfrak{i}$ is nilpotent if $\mathfrak{i}^m=0$ for some $m$
and the minimal such $m$ is called the {\em nilpotency degree} of $\mathfrak{i}$
and denoted $\mathrm{nd}(\mathfrak{i})$.
In the classical case of the Lie algebra $\mathfrak{sl}_n$ there is a beautiful 
bijection between the ad-nilpotent ideals of $\mathfrak{b}$ and Dyck paths, constructed
in \cite{AKOP}, which leads to a combinatorial interpretation of the 
nilpotency degree of an ad-nilpotent ideal in terms of the height of the corresponding
Dyck path. 

To generalize this, for a basic ideal $\mathfrak{i}$ define the {\em quasi-nilpotency degree}
$\mathrm{qnd}(\mathfrak{i})$ of $\mathfrak{i}$ as the minimal non-negative integer $m$ such that 
$\mathfrak{i}^m\subset \mathfrak{a}$. In other words, $\mathrm{qnd}(\mathfrak{i})$
is the nilpotency degree of the image of $\mathfrak{i}$ in $\hat{\mathfrak{n}}_+/\mathfrak{a}$. Then $\mathfrak{i}$ is
quasi-abelian if and only if $\mathrm{qnd}(\mathfrak{i})=1$.
This invariant can be interpreted in combinatorial  terms as follows:

\begin{proposition}\label{prop31}
Let  $\mathfrak{i}$ be a basic ideal and $m=\mathrm{nd}(\mathfrak{i}_+)$. 
\begin{enumerate}[$($a$)$]
\item\label{prop31.1} The number $\mathrm{qnd}(\mathfrak{i})$ equals either $m$ or $m+1$.
\item\label{prop31.2} If $m=0$ then $\mathrm{qnd}(\mathfrak{i})=1$.
\item\label{prop31.3} If $m>0$, then $\mathrm{qnd}(\mathfrak{i})=m$ 
if and only if there does not exist $\beta\in\mathbf{s}_-(\mathfrak{i})$ 
such that  $-\beta\in \mathrm{supp}(\mathfrak{i}_+^{m-1})$.
\end{enumerate}
\end{proposition}

\begin{proof}
Claim \eqref{prop31.2} is clear, so in the rest of the proof we may assume $m>0$.

Obviously, $\mathrm{qnd}(\mathfrak{i})\geq m$. First we observe that if
$\beta\in\Delta_-$ is such that $-\beta\not\in\mathbf{s}_+(\mathfrak{i})$,
then $[\mathfrak{i}_+,\hat{\mathfrak{g}}_{\beta+\delta}]\subset\mathfrak{a}$.
Hence, if $\beta\in\Delta_-$ and $\alpha_1,\alpha_2,\dots,\alpha_k\in 
\mathbf{s}_+(\mathfrak{i})$ are such that 
\begin{displaymath}
[\hat{\mathfrak{g}}_{\alpha_k},[\hat{\mathfrak{g}}_{\alpha_{k-1}},
\dots[\hat{\mathfrak{g}}_{\alpha_1},
\hat{\mathfrak{g}}_{\beta+\delta}]\dots]]\not\in\mathfrak{a},
\end{displaymath}
then all the following elements: 
\begin{displaymath}
-\beta, -(\beta+\alpha_1),\dots, 
-(\beta+\alpha_1+\alpha_{2}+\dots+\alpha_{k-1})
\end{displaymath}
belong to $\mathbf{s}_+(\mathfrak{i})$.
This means that $\hat{\mathfrak{g}}_{\beta}\in \mathfrak{i}_+^{k-1}$,
implying \eqref{prop31.1}.

Conversely, if $\beta\in\Delta_-$ is such that 
$-\beta\in\mathrm{supp}(\mathfrak{i}_+^{m-1})$, then there exist
$\alpha_1,\alpha_2,\dots,\alpha_{m}\in 
\mathbf{s}_+(\mathfrak{i})$ with the following property: 
\begin{displaymath}
\hat{\mathfrak{g}}_{-\beta}=
[\hat{\mathfrak{g}}_{\alpha_{1}},[\hat{\mathfrak{g}}_{\alpha_{2}},
\dots[\hat{\mathfrak{g}}_{\alpha_{m-1}},
\hat{\mathfrak{g}}_{\alpha_{m}}]\dots]].
\end{displaymath}
This implies 
\begin{displaymath}
[\hat{\mathfrak{g}}_{\alpha_{m}},[\hat{\mathfrak{g}}_{\alpha_{m-1}},
\dots[\hat{\mathfrak{g}}_{\alpha_1},
\hat{\mathfrak{g}}_{\beta+\delta}]\dots]]\not\in\mathfrak{a} 
\end{displaymath}
and claim \eqref{prop31.3} follows, completing the proof.
\end{proof}

It would be interesting to have an explicit formula for the number
of basic ideal $\mathfrak{i}$ satisfying $\mathrm{qnd}(\mathfrak{i})=m$,
where $m=1,2,3,\dots$. This extends Problem~\ref{problem28} which
covers the case $m=1$.

\section{Arbitrary ideals and their supports}\label{s4}

\subsection{Supports of arbitrary ideals}\label{s4.1}

Let again $\hat{\mathfrak{g}}$ be an  untwisted affine Lie algebra. In this subsection we describe the support 
of an arbitrary nonzero $\hat{\mathfrak{b}}$-ideal in $\hat{\mathfrak{n}}_+$. We start with adjusting our
previous definition of the support to this more general situation. Let $\mathfrak{i}$ be a 
$\hat{\mathfrak{b}}$-ideal in $\hat{\mathfrak{n}}_+$. 
The {\em support} $\mathrm{supp}(\mathfrak{i})$ is the set
of all $\alpha\in\hat{\Delta}_+$ such that 
$\mathfrak{i}_{\alpha}:=\mathfrak{i}\cap \hat{\mathfrak{g}}_{\alpha}\neq 0$.
The {\em level} of a nonzero ideal $\mathfrak{i}$ is the minimal positive integer $l$ such that 
$l\delta\in\mathrm{supp}(\mathfrak{i})$. Obviously, the level is well-defined
for every nonzero ideal. Our first observation is the following:

\begin{proposition}\label{prop41}
Let $\mathfrak{i}$ be an ideal of level $l$. Then 
we have $\hat{\mathfrak{g}}_{(l+1)\delta}\subset\mathfrak{i}$
and for every $i>l$ and  $\alpha\in D_i$ we have
$\hat{\mathfrak{g}}_{\alpha}\subset \mathfrak{i}$.  
\end{proposition}

\begin{proof}
Let $h\otimes t^l$ be a nonzero element in $\mathfrak{i}_{l\delta}$. Then $h\neq 0$ and hence there exists
a simple root $\beta\in \Delta_{+}$ such that $\beta(h)\neq 0$. As $\mathfrak{i}$ is a 
$\hat{\mathfrak{b}}$-ideal and the root $l\delta+\beta$ is real, this implies 
$\hat{\mathfrak{g}}_{\beta+l\delta}\subset \mathfrak{i}$. Commuting $\hat{\mathfrak{g}}_{\beta+l\delta}$
with $\hat{\mathfrak{g}}_{\gamma}$ for suitable $\gamma\in \Delta_+$, we obtain 
$\hat{\mathfrak{g}}_{\xi+l\delta}\subset \mathfrak{i}$ for any $\xi\in \Delta_+$, $\xi\geq \beta$.
The complement to the set of all such $\xi$ belongs to the hyperplane spanned by all simple roots of
$\Delta_+$ different from $\beta$. This implies that the set of all such $\xi$ spans $\mathfrak{h}^*$.
It follows that the commutants of the form
\begin{displaymath}
[\hat{\mathfrak{g}}_{\xi+l\delta},\hat{\mathfrak{g}}_{-\xi+(l+1)\delta}],
\end{displaymath}
where $\xi$ is as above, span $\hat{\mathfrak{g}}_{(l+1)\delta}$. Hence
$\hat{\mathfrak{g}}_{(l+1)\delta}\subset \mathfrak{i}$. This implies that
$\hat{\mathfrak{g}}_{\beta+(l+1)\delta}\subset \mathfrak{i}$ for every real $\beta\in D$ and the proof is
completed by induction.
\end{proof}

As an immediate corollary we obtain:

\begin{corollary}\label{cor42}
Every nonzero $\hat{\mathfrak{b}}$-ideal in $\hat{\mathfrak{n}}_+$ has finite codimension.
\end{corollary}

For a nonzero $\hat{\mathfrak{b}}$-ideal $\mathfrak{i}$ of level $l$ 
in $\hat{\mathfrak{n}}_+$ set:
\begin{displaymath}
\begin{array}{ccl}
\mathtt{a}_+(\mathfrak{i})&:=& \{\alpha\in\Delta_+:\alpha+(l-1)\delta\in \mathrm{supp}(\mathfrak{i})\};\\
\mathtt{a}_-(\mathfrak{i})&:=& \{\alpha\in\Delta_-:\alpha+l\delta\in \mathrm{supp}(\mathfrak{i})\};\\
\mathtt{a}'_+(\mathfrak{i})&:=& \{\alpha\in\Delta_+:\alpha+l\delta\in \mathrm{supp}(\mathfrak{i})\};\\
\mathtt{a}'_-(\mathfrak{i})&:=& \{\alpha\in\Delta_-:\alpha+(l+1)\delta\in \mathrm{supp}(\mathfrak{i})\}.
\end{array}
\end{displaymath}

These sets determine the support of $\mathfrak{i}$ in the following sense:

\begin{proposition}\label{prop43}
Let $\mathfrak{i}$ be an ideal of level $l$. Then we have:
\begin{multline*}
\mathrm{supp}(\mathfrak{i})=\{l\delta,(l+1)\delta\}\cup\{\alpha+(l-1)\delta:\alpha\in \mathtt{a}_+(\mathfrak{i})\}
\cup\\ \cup\{\alpha+l\delta:\alpha\in \mathtt{a}_-(\mathfrak{i})\cup\mathtt{a}'_+(\mathfrak{i})\}
\cup\\ \cup\{\alpha+(l+1)\delta:\alpha\in \mathtt{a}'_-(\mathfrak{i})\}\cup
\bigcup_{i\geq l+1}D_i.
\end{multline*}
\end{proposition}

\begin{proof}
That $\mathrm{supp}(\mathfrak{i})$ contains the right hand side follows from Proposition~\ref{prop41}
and definitions. To prove the reverse inclusion it is enough to show that $\mathrm{supp}(\mathfrak{i})$
does not intersect $D_i$ for $i<l-1$. Assume that $\beta\in \mathrm{supp}(\mathfrak{i})\cap D_i$ for some
$i<l-1$. Then $\beta$ is a real root as $\mathfrak{i}$ has level $l$.  If $\beta=\gamma+i\delta$ for some
$\gamma\in\Delta_+$, then, commuting $\hat{\mathfrak{g}}_{\beta}$ with $\hat{\mathfrak{g}}_{-\beta+\delta}$
we get $(i+1)\delta\in \mathrm{supp}(\mathfrak{i})$, a contradiction. If $\beta=\gamma+(i+1)\delta$ for some
$\gamma\in\Delta_-$, then, commuting $\hat{\mathfrak{g}}_{\beta}$ with $\hat{\mathfrak{g}}_{-\gamma}$
we again get $(i+1)\delta\in \mathrm{supp}(\mathfrak{i})$, a contradiction. The claim follows.
\end{proof}

For a nonzero $\hat{\mathfrak{b}}$-ideal $\mathfrak{i}$ in $\hat{\mathfrak{n}}_+$ set
$\mathfrak{i}_+:=\oplus_{\alpha\in \mathtt{a}_+(\mathfrak{i})} \hat{\mathfrak{g}}_{\alpha}$
and define $\mathfrak{i}'_+$ similarly (using $\mathtt{a}'_+(\mathfrak{i})$). 
Then both $\mathfrak{i}_+$ and $\mathfrak{i}'_+$ are 
$\mathfrak{b}$-ideals of $\mathfrak{n}_+$. Denote by $\mathfrak{i}_-$ the 
$\mathfrak{b}$-submodule of $\mathfrak{g}/\mathfrak{b}$ which is canonically identified with 
$\oplus_{\alpha\in \mathtt{a}_-(\mathfrak{i})} \hat{\mathfrak{g}}_{\alpha}$
and define $\mathfrak{i}'_-$ similarly (using $\mathtt{a}'_-(\mathfrak{i})$).
From Proposition~\ref{prop43} it follows that the quadruple 
\begin{displaymath}
\underline{\mathfrak{i}}:=
(\mathfrak{i}_+,\mathfrak{i}_-,\mathfrak{i}'_+,\mathfrak{i}'_-) 
\end{displaymath}
determines $\mathrm{supp}(\mathfrak{i})$. Taking the $\mathfrak{h}$-complements to $\mathfrak{i}_-$
and $\mathfrak{i}'_-$ in $\hat{\mathfrak{n}}_-$ and applying the Chevalley involution we obtain that
$\mathrm{supp}(\mathfrak{i})$ is uniquely determined by a quadruple of $\mathfrak{b}$-ideals in
$\mathfrak{n}$. This implies:

\begin{corollary}\label{cor44}
The cardinality of the set of all possible supports for 
nonzero $\hat{\mathfrak{b}}$-ideals of level $l$ in $\hat{\mathfrak{n}}_+$ does not exceed the
fourth power of the number of $\mathfrak{b}$-ideals in $\mathfrak{n}_+$.
\end{corollary}

As expected, not every quadruple can appear as $\underline{\mathfrak{i}}$
for some nonzero $\hat{\mathfrak{b}}$-ideal of level $l$ 
in $\hat{\mathfrak{n}}_+$. Here are some natural restrictions:

\begin{proposition}\label{prop45}
Let ${\mathfrak{i}}$ be a nonzero $\hat{\mathfrak{b}}$-ideal 
of level $l$ in $\hat{\mathfrak{n}}_+$.
\begin{enumerate}[$($a$)$]
\item\label{prop45.1} If $\mathfrak{i}_+\neq 0$, then
$\mathfrak{i}'_-=\mathfrak{g}/\mathfrak{b}$.
\item\label{prop45.2} Both $\mathfrak{i}'_+$ and $\mathfrak{i}'_-$
are nonzero.
\item\label{prop45.3}
If $\mathtt{a}_-(\mathfrak{i})$ contains $-\alpha$
for all simple $\alpha\in\Delta_+$, then $\mathfrak{i}'_+=\mathfrak{n}_+$. 
\item\label{prop45.4}
If $\mathfrak{i}_+=\mathfrak{n}_+$, then 
$\mathtt{a}_-(\mathfrak{i})$ equals either
$\Delta_-$ or $\Delta_-\setminus\{-\alpha_{\mathrm{max}}\}$.
If $\mathfrak{i}'_+=\mathfrak{n}_+$, then 
$\mathtt{a}'_-(\mathfrak{i})$ equals either
$\Delta_-$ or $\Delta_-\setminus\{-\alpha_{\mathrm{max}}\}$.
\end{enumerate}
\end{proposition}

\begin{proof}
If $\mathfrak{i}_+\neq 0$, then $\alpha_{\mathrm{max}}\in
\mathtt{a}_+(\mathfrak{i})$. Applying twice the adjoint action of
$\hat{\mathfrak{g}}_{-\alpha_{\mathrm{max}}+\delta}$ to 
$\hat{\mathfrak{g}}_{\alpha_{\mathrm{max}}+(l-1)\delta}$
we get a nonzero space and thus $-\alpha_{\mathrm{max}}\in
\mathtt{a}'_-(\mathfrak{i})$. The latter implies
$\mathtt{a}'_-(\mathfrak{i})=\Delta_-$ and claim \eqref{prop45.1} follows.

Let $h\otimes t^l$ be a nonzero element in 
$\mathfrak{i}_{l\delta}$. Then there exists $\alpha\in\Delta_+$
such that $\alpha(h)\neq 0$, which implies 
$\alpha\in \mathtt{a}'_+(\mathfrak{i})$ and
$-\alpha\in \mathtt{a}'_-(\mathfrak{i})$. Claim \eqref{prop45.2} follows.
 
If $\mathtt{a}_-(\mathfrak{i})$ contains $-\alpha$
for all simple $\alpha$, then commutations with the corresponding
$\hat{\mathfrak{g}}_{\alpha}$'s imply 
$\hat{\mathfrak{g}}_{l\delta}\subset\mathfrak{i}$. Now claim  
\eqref{prop45.3} follows using commutation with 
$\hat{\mathfrak{g}}_{\beta}$, $\beta\in\Delta_+$.

If $\beta\in\Delta_+$ is not maximal, then there exists a simple
root $\alpha\in\Delta_+$ such that $\alpha+\beta\geq \beta$.
Commuting $\hat{\mathfrak{g}}_{\alpha+(l-1)\delta}$ with
$\hat{\mathfrak{g}}_{-\alpha-\beta+\delta}$ gives that 
$\hat{\mathfrak{g}}_{-\beta+l\delta}\subset \mathfrak{i}$.
This proves the first part of claim \eqref{prop45.4} and the
second part is proved similarly.
\end{proof}

A more detailed description of possible supports in type $A$
is given in Subsection~\ref{s4.3} below.
                                                                          
\subsection{Support equivalent ideals}\label{s4.2}

Two nonzero $\hat{\mathfrak{b}}$-ideals $\mathfrak{i}$
and $\mathfrak{j}$ in $\hat{\mathfrak{n}}_+$ will be called
{\em support equivalent} provided that 
$\mathrm{supp}(\mathfrak{i})=\mathrm{supp}(\mathfrak{j})+k\delta$
for some $k\in\mathbb{Z}$. This notion is motivated by the following
statement:

\begin{proposition}\label{prop51}
For $l\in\mathbb{N}$ let $\mathcal{S}_l$ denote the set of all
possible supports of $\hat{\mathfrak{b}}$-ideals of level $l$
in $\hat{\mathfrak{n}}_+$. Then the map $X\mapsto X+\delta$
is a bijection between $\mathcal{S}_l$ and $\mathcal{S}_{l+1}$. 
\end{proposition}

\begin{proof}
If $\mathfrak{i}$ is an ideal of level $l$, define the ideal
$\mathfrak{j}$ of level $l+1$ as follows: $\mathfrak{j}$ is generated,
as a vector space, by all elements of the form $x\otimes t^{i+1}$ whenever
$x\in\mathfrak{g}$ is such that and $x\otimes t^{i}\in \mathfrak{i}$.

Conversely, if $\mathfrak{j}$ is an ideal of level $l+1$, define the ideal
$\mathfrak{i}$ of level $l$ as follows: $\mathfrak{i}$ is generated,
as a vector space, by all elements of the form $x\otimes t^{i}$ whenever
$x\in\mathfrak{g}$ is such that and $x\otimes t^{i+1}\in \mathfrak{j}$.

It is easy to check that these maps are mutually inverse bijections
between the sets of ideals of level $l$ and level $l+1$. By
construction, we also have 
$\mathrm{supp}(\mathfrak{j})=\mathrm{supp}(\mathfrak{i})+\delta$.
The claim follows.
\end{proof}

From Corollary~\ref{cor44} it follows that the number of equivalence
classes of support equivalent ideals is finite. It would be interesting
to have an explicit combinatorial formula for this number for all
untwisted affine Lie algebras.

\subsection{Support equivalent ideals in type $A$}\label{s4.3}

We go back to the special case $\hat{\mathfrak{g}}=
\hat{\mathfrak{sl}}_n$. From Proposition~\ref{prop51} it follows
that in order to understand supports of $\hat{\mathfrak{b}}$-ideal
in $\hat{\mathfrak{n}}_+$ it is enough to understand elements
of $\mathcal{S}_1$. We will need the following definition:
for a Dyck path $p\in\mathcal{D}_n$ denote by $\mathtt{m}(p)$
the set of all points $(2m,0)$, $m=1,2,\dots,n-1$, for which
\begin{displaymath}
\{(2m-2,0),(2m,0),(2m+2,0)\}\setminus 
(\mathtt{v}_{(0)}(p)\cup\{(0,0),(2n,0)\})
\neq\varnothing. 
\end{displaymath}

Recall that 
$\alpha_1,\dots,\alpha_{n-1}$ are the usual simple roots of 
$\mathfrak{sl}_n$ (the root $\alpha_i$ corresponds to the matrix
unit $e_{i,i+1}$). 
Let $\{h_1,\dots,h_{n-1}\}$ be the basis of $\mathfrak{h}$,
dual to $\{\alpha_1,\dots,\alpha_{n-1}\}$,
that is $\alpha_j(h_i)=\delta_{i,j}$.
With every ideal $\mathfrak{i}$ of level $1$ we associate the
quadruple $\Psi(\mathfrak{i})=(p,q,p',q')$ of Dyck paths as follows:
$p$ and $p'$ are associated to $\mathfrak{i}_+$ and $\mathfrak{i}'_+$
(or, rather, $\mathtt{a}_+(\mathfrak{i})$ and $\mathtt{a}'_+(\mathfrak{i})$),
respectively, and $q$ and $q'$ are associated to $\mathfrak{i}_-$ 
and $\mathfrak{i}'_-$ (or, rather, $\mathtt{a}_-(\mathfrak{i})$ 
and $\mathtt{a}'_-(\mathfrak{i})$), respectively, as described in Subsection~\ref{s3.3}. The quadruple $(p,q,p',q')$ defines the 
support equivalence class of the ideal
$\mathfrak{i}$ uniquely, however, not every quadruple of Dyck paths
appears as $\Psi(\mathfrak{i})$ for some $\mathfrak{i}$.
For $n=1$ we have a unique quadruple $(p,q,p',q')=
(\mathbf{p},\mathbf{p},\mathbf{p},\mathbf{p})$ and the
corresponding unique equivalence class of support equivalent ideals.
Our main result in this subsection is the following:

\begin{theorem}\label{thm55}
Assume that $n>1$. A quadruple $(p,q,p',q')$ of Dyck paths has the form
$\Psi(\mathfrak{i})$ for some $\hat{\mathfrak{b}}$-ideal
of level $1$ in $\hat{\mathfrak{n}}_+$ if and only if one of 
the following (mutually excluding) conditions is satisfied:
\begin{enumerate}[$($i$)$]
\item\label{thm55.1} $p=q'=\mathbf{p}$, $q=\mathbf{q}$ 
and $\mathbf{v}_{(0)}(p')=1$;
\item\label{thm55.2} $p=\mathbf{p}$, $q=\mathbf{q}$, 
$\mathbf{v}_{(0)}(p')>1$ and $q'\geq \overline{p'}$;
\item\label{thm55.3} $p=\mathbf{p}$, $q\neq\mathbf{q}$, 
$\mathtt{m}(q)\subset\mathbf{v}_{(0)}(p')$, $q'\geq \overline{p'}$
and, additionally, $q'=\mathbf{p}$ if $(2,0)\not\in
\mathtt{v}_{(0)}(q)$ or $(2n-2,0)\not\in
\mathtt{v}_{(0)}(q)$;
\item\label{thm55.4} $p\neq \mathbf{p}$, $q\geq \overline{p}$,
$\mathtt{m}(q)\cup \{(2,0),(2n-2,0)\}\subset\mathbf{v}_{(0)}(p')$, 
and $q'=\mathbf{p}$.
\end{enumerate}
\end{theorem}

\begin{proof}
If $\Psi(\mathfrak{i})=(\mathbf{p},\mathbf{q},p',q')$, then 
$h\otimes t\in \mathfrak{i}$ for some nonzero $h\in\mathfrak{h}$
and hence $\alpha_i(h)\neq 0$ for some $i$, implying
$\alpha_i\in\mathtt{a}'_+(\mathfrak{i})$ and so 
$\mathtt{v}_{(0)}(p')\neq\varnothing$.
Assume first that $\Psi(\mathfrak{i})=(p,q,p',q')$ and that 
\begin{equation}\label{eqn1}
p=\mathbf{p},\, q=\mathbf{q},\,\mathtt{v}_{(0)}(p')=\{(2m,0)\},
m\in\{1,2,\dots,n-1\}.
\end{equation}
Then we have $\mathfrak{i}_+=\mathfrak{i}_-=0$,
moreover, $\mathfrak{i}_{\delta}$ has to be one-dimensional 
and generated by $h_m\otimes t$. As 
$\alpha_{\mathrm{max}}(h_m)=1\neq 0$, it follows that 
$-\alpha_{\mathrm{max}}\in \mathtt{a}'_-(\mathfrak{i})$,
which implies $q'=\mathbf{p}$. On the other hand, given 
$(p,q,p',q')$ satisfying \eqref{thm55.1}, let $\mathfrak{j}$
denote the subspace of $\hat{\mathfrak{n}}_+$, spanned
by $h_m\otimes t$ and all $\hat{\mathfrak{g}}_{\alpha}$
such that $\alpha\in D_i$, $i>1$, or $\alpha\in D_1$ 
is such that $\alpha$ does not have the form 
$\beta+\delta$ for some $\beta\in\Delta_+$ for which the 
corresponding box is to the north-east of  $p'$ 
in the matrix setup of Subsection~\ref{s3.3} (see Figure~\ref{fig2}).
It is easy to see that $\mathfrak{j}$ is an ideal
and that $\Psi(\mathfrak{j})=(p,q,p',q')$, which
proves our theorem for quadruples satisfying
\eqref{eqn1}. The number of such quadruples equals the number
of Dyck paths of semilength $n$ having a unique valley of height $0$. 
It is well-known that this number is $C_{n-1}$.

Assume now that $\Psi(\mathfrak{i})=(p,q,p',q')$ and that
\begin{equation}\label{eqn2}
p=\mathbf{p},\, q=\mathbf{q},\,\mathbf{v}_{(0)}(p')>1.
\end{equation}
Then from the proof of Lemma~\ref{lem9} 
we get $q'\geq \overline{p'}$. 
On the other hand,
assume that $(p,q,p',q')$ satisfies \eqref{thm55.2}.
Let $(2m,0)$ and $(2k,0)$ be two different valleys of
$p'$. Denote by $\mathfrak{j}$ the subspace of $\hat{\mathfrak{n}}_+$, 
spanned by $(h_m-h_k)\otimes t$ and all $\hat{\mathfrak{g}}_{\alpha}$
such that $\alpha\in D_i$, $i>1$, or $\alpha\in D_1$ 
is such that  either  
$\alpha=\beta+\delta$ for some $\beta\in\Delta_+$ for which the 
corresponding box is to the north-east of $p'$ 
in the matrix setup of Subsection~\ref{s3.3} 
or $\alpha=\beta+2\delta$ 
for some $\beta\in\Delta_-$ for which the 
corresponding box is to the north-east of $q'$.
Note that if $\beta\in \Delta_-$ corresponds to a box lying to
the south-west of $\overline{p}$, then $\beta(h_m-h_k)=0$.
Using this it is easy to check that $\mathfrak{j}$ is an ideal
and  $\Psi(\mathfrak{j})=(p,q,p',q')$, which
proves our theorem for quadruples satisfying \eqref{eqn2}.

Assume now that $\Psi(\mathfrak{i})=(p,q,p',q')$ and that
\begin{equation}\label{eqn3}
p=\mathbf{p},\, q\neq \mathbf{q}.
\end{equation}
From the proof of Lemma~\ref{lem9} we get $\overline{p'}\leq q'$. 
If $(2m,0)\not\in\mathtt{v}_{(0)}(q)$ 
for some $m\in\{1,\dots,n-1\}$, then $\hat{\mathfrak{g}}_{-\alpha_m+\delta}\subset
\mathfrak{i}$. Applying twice the adjoint action of
$\hat{\mathfrak{g}}_{\alpha_m}$ we get that 
$\hat{\mathfrak{g}}_{\alpha_m+\delta}\subset
\mathfrak{i}$, which yields 
$(2m,0)\in \mathtt{v}_{(0)}(p')$.
If $(2m-2,0)\not\in\mathtt{v}_{(0)}(q)$ 
for some $m\in\{2,\dots,n-1\}$, then $\hat{\mathfrak{g}}_{-\alpha_{m-1}+\delta}\subset
\mathfrak{i}$. Commuting this with
$\hat{\mathfrak{g}}_{\alpha_{m-1}}$ and then with
$\hat{\mathfrak{g}}_{\alpha_{m}}$ we get that 
$\hat{\mathfrak{g}}_{\alpha_m+\delta}\subset
\mathfrak{i}$, which yields that $(2m,0)\in \mathtt{v}_{(0)}(p')$.
If $(2m+2,0)\not\in\mathtt{v}_{(0)}(q)$ 
for some $m\in\{1,\dots,n-2\}$, then $\hat{\mathfrak{g}}_{-\alpha_{m+1}+\delta}\subset
\mathfrak{i}$. Commuting this with
$\hat{\mathfrak{g}}_{\alpha_{m+1}}$ and then with
$\hat{\mathfrak{g}}_{\alpha_{m}}$ we get that 
$\hat{\mathfrak{g}}_{\alpha_m+\delta}\subset
\mathfrak{i}$, which yields that $(2m,0)\in \mathtt{v}_{(0)}(p')$.
If $(2,0)\not\in\mathtt{v}_{(0)}(q)$, then $\mathfrak{i}$
contains $\hat{\mathfrak{g}}_{-\alpha_{1}+\delta}$ and hence
also 
\begin{displaymath}
[\hat{\mathfrak{g}}_{-\alpha_{1}+\delta},
\hat{\mathfrak{g}}_{-\alpha_{\mathrm{max}}+\alpha_{1}+\delta}]=
\hat{\mathfrak{g}}_{-\alpha_{\mathrm{max}}+2\delta},
\end{displaymath}
which implies $q'=\mathbf{p}$. Similarly for
$(2n-2,0)\not\in\mathtt{v}_{(0)}(q)$. This establishes necessity
of condition \eqref{thm55.3}.

On the other hand, to prove sufficiency
assume that $(p,q,p',q')$ satisfies \eqref{thm55.3}.
Denote by $\mathfrak{j}$ the subspace of $\hat{\mathfrak{n}}_+$, 
spanned by 
\begin{itemize}
\item the space $V_1$ generated by 
$\hat{\mathfrak{g}}_{\alpha+\delta}$ for every 
$\alpha\in \Delta_-$ corresponding to a box lying to the north-east 
of $q$ in the matrix setup of Subsection~\ref{s3.3};
\item the space $V_2$ generated by $[\hat{\mathfrak{g}}_{\alpha+\delta},
\hat{\mathfrak{g}}_{-\alpha}]$ for every $\alpha\in \Delta_-$
corresponding to a box lying to the north-east of $q$;
\item the space $V_3$ generated by 
$\hat{\mathfrak{g}}_{\alpha+\delta}$ for every 
$\alpha\in \Delta_+$ corresponding to a box lying to the north-east 
of $p'$;
\item the space $V_4$ generated by  
$\hat{\mathfrak{g}}_{\alpha+2\delta}$ for every 
$\alpha\in \Delta_-$ corresponding to a box lying to the north-east 
of $q'$;
\item the space $V_5$ generated by $\hat{\mathfrak{g}}_{2\delta}$
and all  $\hat{\mathfrak{g}}_{\beta}$,
$\beta\in D_i$, $i>1$.
\end{itemize}
We claim that  $\mathfrak{j}$ is an ideal. The essential combinatorics 
of the proof is shown in Figure~\ref{fig11}. 
That $V_3\oplus V_4\oplus V_5$ is an ideal follows from the proof
of Lemma~\ref{lem9}. By construction of $V_2$, the adjoint action of
$\mathfrak{n}_+$ applied to $V_1$ or $V_1\oplus V_2$ generates
the same subspace of $\hat{\mathfrak{g}}$, call it $X$. Let
$\mathtt{a}(X)$ be the set of all $\alpha\in\Delta_+$
such that $\hat{\mathfrak{g}}_{\alpha+\delta}\subset X$.
Then $\mathtt{a}(X)$ is the support of an ad-nilpotent ideal
in $\mathfrak{b}$. Let $w$ be the corresponding Dyck path.
Using a computation similar to the one used in
the previous paragraph one checks that
$w$ is the maximum (with respect to $\leq$) path satisfying 
$\mathtt{m}(q)\subset\mathbf{v}_{(0)}(w)$
(this is shown by the line $l$ in Figure~\ref{fig11}). As
$\mathtt{m}(q)\subset\mathbf{v}_{(0)}(p')$, we have $p'\leq w$
and it follows that $V_1\oplus V_2\oplus V_3$ is stable under the 
adjoint action of $\mathfrak{b}$. It remains to show that 
for $\alpha\in\Delta_-$ the adjoint action of 
$\hat{\mathfrak{g}}_{\alpha+\delta}$
maps $V_1$ to $V_4\oplus V_5$. This is clear in the cases
$(2,0)\in\mathtt{v}_{(0)}(q)$ or $(2n-2,0)\in\mathtt{v}_{(0)}(q)$.
In other cases this follows from Lemma~\ref{lem9} 
using the condition $\overline{p'}\leq q'$. 
By construction, $\Psi(\mathfrak{j})=(p,q,p',q')$.
This proves 
our theorem for all quadruples  satisfying \eqref{eqn3}.

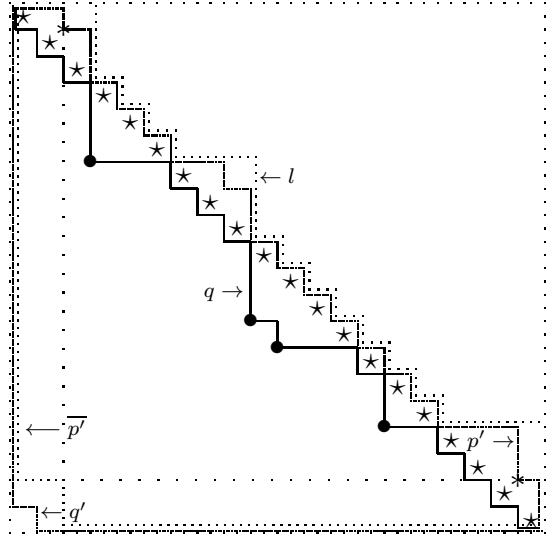
\begin{figure}
\begin{picture}(220.00,220.00)
\dottedline{5}(10.00,10.00)(10.00,210.00)
\dottedline{5}(210.00,210.00)(10.00,210.00)
\dottedline{5}(210.00,210.00)(210.00,10.00)
\dottedline{5}(10.00,10.00)(210.00,10.00)
%%%%%%%%%%%%%%%%%%%%%%%%%%%%%%%%%%
\put(15.00,205.00){\makebox(0,0)[cc]{$\star$}}
\put(25.00,195.00){\makebox(0,0)[cc]{$\star$}}
\put(35.00,185.00){\makebox(0,0)[cc]{$\star$}}
\put(45.00,175.00){\makebox(0,0)[cc]{$\star$}}
\put(55.00,165.00){\makebox(0,0)[cc]{$\star$}}
\put(65.00,155.00){\makebox(0,0)[cc]{$\star$}}
\put(75.00,145.00){\makebox(0,0)[cc]{$\star$}}
\put(85.00,135.00){\makebox(0,0)[cc]{$\star$}}
\put(95.00,125.00){\makebox(0,0)[cc]{$\star$}}
\put(105.00,115.00){\makebox(0,0)[cc]{$\star$}}
\put(115.00,105.00){\makebox(0,0)[cc]{$\star$}}
\put(125.00,95.00){\makebox(0,0)[cc]{$\star$}}
\put(135.00,85.00){\makebox(0,0)[cc]{$\star$}}
\put(145.00,75.00){\makebox(0,0)[cc]{$\star$}}
\put(155.00,65.00){\makebox(0,0)[cc]{$\star$}}
\put(165.00,55.00){\makebox(0,0)[cc]{$\star$}}
\put(175.00,45.00){\makebox(0,0)[cc]{$\star$}}
\put(185.00,35.00){\makebox(0,0)[cc]{$\star$}}
\put(195.00,25.00){\makebox(0,0)[cc]{$\star$}}
\put(205.00,15.00){\makebox(0,0)[cc]{$\star$}}
%%%%%%%%%%%%%%%%%%%%%%%%%%%%%%%%%%
\drawline(12.00,208.00)(12.00,200.00)
\drawline(20.00,200.00)(12.00,200.00)
\drawline(20.00,200.00)(20.00,190.00)
\drawline(30.00,190.00)(20.00,190.00)
\drawline(30.00,190.00)(30.00,180.00)
\drawline(40.00,180.00)(30.00,180.00)
\drawline(40.00,180.00)(40.00,150.00)
\drawline(70.00,150.00)(40.00,150.00)
\drawline(70.00,150.00)(70.00,140.00)
\drawline(80.00,140.00)(70.00,140.00)
\drawline(80.00,140.00)(80.00,130.00)
\drawline(90.00,130.00)(80.00,130.00)
\drawline(90.00,130.00)(90.00,120.00)
\drawline(100.00,120.00)(90.00,120.00)
\drawline(100.00,120.00)(100.00,90.00)
\drawline(100.00,120.00)(100.00,90.00)
\drawline(110.00,90.00)(100.00,90.00)
\drawline(110.00,90.00)(110.00,80.00)
\drawline(140.00,80.00)(110.00,80.00)
\drawline(140.00,80.00)(140.00,70.00)
\drawline(150.00,70.00)(140.00,70.00)
\drawline(150.00,70.00)(140.00,70.00)
\drawline(150.00,70.00)(150.00,50.00)
\drawline(170.00,50.00)(150.00,50.00)
\drawline(170.00,50.00)(170.00,40.00)
\drawline(180.00,40.00)(170.00,40.00)
\drawline(180.00,40.00)(180.00,30.00)
\drawline(190.00,30.00)(180.00,30.00)
\drawline(190.00,30.00)(190.00,20.00)
\drawline(200.00,20.00)(190.00,20.00)
\drawline(200.00,20.00)(200.00,12.00)
\drawline(208.00,12.00)(200.00,12.00)
%%%%%%%%%%%%%%%%%%%%%%%%%%%%%%%
\dottedline{1}(12.00,208.00)(30.00,208.00)
\dottedline{1}(30.00,200.00)(30.00,208.00)
\dottedline{1}(30.00,200.00)(40.00,200.00)
\dottedline{1}(40.00,180.00)(40.00,200.00)
\dottedline{1}(40.00,180.00)(50.00,180.00)
\dottedline{1}(50.00,170.00)(50.00,180.00)
\dottedline{1}(50.00,170.00)(60.00,170.00)
\dottedline{1}(60.00,160.00)(60.00,170.00)
\dottedline{1}(60.00,160.00)(70.00,160.00)
\dottedline{1}(70.00,150.00)(70.00,160.00)
\dottedline{1}(70.00,150.00)(90.00,150.00)
\dottedline{1}(90.00,140.00)(90.00,150.00)
\dottedline{1}(90.00,140.00)(100.00,140.00)
\dottedline{1}(100.00,120.00)(100.00,140.00)
\dottedline{1}(100.00,120.00)(110.00,120.00)
\dottedline{1}(110.00,110.00)(110.00,120.00)
\dottedline{1}(110.00,110.00)(120.00,110.00)
\dottedline{1}(120.00,100.00)(120.00,110.00)
\dottedline{1}(120.00,100.00)(130.00,100.00)
\dottedline{1}(130.00,90.00)(130.00,100.00)
\dottedline{1}(130.00,90.00)(140.00,90.00)
\dottedline{1}(140.00,80.00)(140.00,90.00)
\dottedline{1}(140.00,80.00)(150.00,80.00)
\dottedline{1}(150.00,70.00)(150.00,80.00)
\dottedline{1}(150.00,70.00)(160.00,70.00)
\dottedline{1}(160.00,60.00)(160.00,70.00)
\dottedline{1}(160.00,60.00)(170.00,60.00)
\dottedline{1}(170.00,50.00)(170.00,60.00)
\dottedline{1}(170.00,50.00)(200.00,50.00)
\dottedline{1}(200.00,30.00)(200.00,50.00)
\dottedline{1}(200.00,30.00)(208.00,30.00)
\dottedline{1}(208.00,12.00)(208.00,30.00)
%%%%%%%%%%%%%%%%%%%%%%%%%%%%%%%
\dottedline{1}(11.00,209.00)(11.00,20.00)
\dottedline{1}(20.00,20.00)(11.00,20.00)
\dottedline{1}(20.00,20.00)(20.00,11.00)
\dottedline{1}(209.00,11.00)(20.00,11.00)
%%%%%%%%%%%%%%%%%%%%%%%%%%%%%%%%%%
\dottedline{7}(30.00,200.00)(30.00,30.00)
\dottedline{7}(200.00,30.00)(30.00,30.00)
\dottedline{3}(13.00,208.00)(13.00,30.00)
\dottedline{3}(30.00,30.00)(13.00,30.00)
\dottedline{3}(30.00,30.00)(30.00,13.00)
\dottedline{3}(208.00,13.00)(30.00,13.00)
%%%%%%%%%%%%%%%%%%%%%%%%%%%%%%%%%%
\dottedline{3}(42.00,209.00)(42.00,182.00)
\dottedline{3}(52.00,182.00)(42.00,182.00)
\dottedline{3}(52.00,182.00)(52.00,172.00)
\dottedline{3}(62.00,172.00)(52.00,172.00)
\dottedline{3}(62.00,172.00)(62.00,162.00)
\dottedline{3}(72.00,162.00)(62.00,162.00)
\dottedline{3}(72.00,162.00)(72.00,152.00)
\dottedline{3}(102.00,152.00)(72.00,152.00)
\dottedline{3}(102.00,152.00)(102.00,122.00)
\dottedline{3}(112.00,122.00)(102.00,122.00)
\dottedline{3}(112.00,122.00)(112.00,112.00)
\dottedline{3}(122.00,112.00)(112.00,112.00)
\dottedline{3}(122.00,112.00)(122.00,102.00)
\dottedline{3}(132.00,102.00)(122.00,102.00)
\dottedline{3}(132.00,102.00)(132.00,92.00)
\dottedline{3}(142.00,92.00)(132.00,92.00)
\dottedline{3}(142.00,92.00)(142.00,82.00)
\dottedline{3}(152.00,82.00)(142.00,82.00)
\dottedline{3}(152.00,82.00)(152.00,72.00)
\dottedline{3}(162.00,72.00)(152.00,72.00)
\dottedline{3}(162.00,72.00)(162.00,62.00)
\dottedline{3}(172.00,62.00)(162.00,62.00)
\dottedline{3}(172.00,62.00)(172.00,52.00)
\dottedline{3}(208.00,52.00)(172.00,52.00)
%%%%%%%%%%%%%%%%%%%%%%%%%%%%%%%%%%
\put(90.00,100.00){\makebox(0,0)[cc]{\tiny $q\rightarrow$}}
\put(190.00,45.00){\makebox(0,0)[cc]{\tiny $p'\rightarrow$}}
\put(30.00,18.00){\makebox(0,0)[cc]{\tiny $\leftarrow q'$}}
\put(27.00,50.00){\makebox(0,0)[cc]{\tiny $\longleftarrow \overline{p'}$}}
\put(110.00,145.00){\makebox(0,0)[cc]{\tiny $\leftarrow l$}}
\put(30.00,200.00){\makebox(0,0)[cc]{$\ast$}}
\put(200.00,30.00){\makebox(0,0)[cc]{$\ast$}}
\put(40.00,150.00){\makebox(0,0)[cc]{$\bullet$}}
\put(100.00,90.00){\makebox(0,0)[cc]{$\bullet$}}
\put(110.00,80.00){\makebox(0,0)[cc]{$\bullet$}}
\put(150.00,50.00){\makebox(0,0)[cc]{$\bullet$}}
\end{picture}
\caption{Combinatorics of Theorem~\ref{thm55}\eqref{thm55.3}: 
the path $p'$ must 
contain the area on the upper right side of $l$
(which is determined by $\bullet$ of  $q$) and the path $q'$ 
must contain the area on the upper right side of $\overline{p'}$
(which is determined by $\ast$ of  $p'$)}\label{fig11} 
\end{figure}

Finally, assume that $\Psi(\mathfrak{i})=(p,q,p',q')$ and that
\begin{equation}\label{eqn4}
p\neq \mathbf{p}.
\end{equation}
Then $\overline{p}\leq q$ by Lemma~\ref{lem9} and
$q'=\mathbf{p}$ by Proposition~\ref{prop45}\eqref{prop45.1}.
Further, $\hat{\mathfrak{g}}_{\alpha_{\mathrm{max}}}\subset \mathfrak{i}$,
which implies that 
\begin{displaymath}
[\hat{\mathfrak{g}}_{\alpha_{\mathrm{max}}},
\hat{\mathfrak{g}}_{-\alpha_{\mathrm{max}}+\alpha_{1}+\delta}]= 
\hat{\mathfrak{g}}_{\alpha_{1}+\delta}\subset \mathfrak{i} 
\end{displaymath}
and
\begin{displaymath}
[\hat{\mathfrak{g}}_{\alpha_{\mathrm{max}}},
\hat{\mathfrak{g}}_{-\alpha_{\mathrm{max}}+\alpha_{n-1}+\delta}]= 
\hat{\mathfrak{g}}_{\alpha_{n-1}+\delta}\subset \mathfrak{i}.
\end{displaymath}
Hence $\{(2,0),(2n-2,0)\}\subset \mathtt{v}_{(0)}(p')$.
That $\mathtt{m}(q)\subset \mathtt{v}_{(0)}(p')$
is proved similarly to the previous case.
This establishes necessity of \eqref{thm55.4}.

To prove sufficiency
assume that $(p,q,p',q')$ satisfies \eqref{thm55.4}.
Denote by $\mathfrak{j}$ the subspace of $\hat{\mathfrak{n}}_+$, 
spanned by 
\begin{itemize}
\item the space $V_1$ generated by 
$\hat{\mathfrak{g}}_{\alpha}$ for every 
$\alpha\in \Delta_+$ corresponding to a box lying to the north-east 
of $p$ in the matrix setup of Subsection~\ref{s3.3};
\item the space $V_2$ generated by 
$\hat{\mathfrak{g}}_{\alpha+\delta}$ for every 
$\alpha\in \Delta_-$ corresponding to a box lying to the north-east 
of $q$;
\item the space $V_3$ generated by $[\hat{\mathfrak{g}}_{\alpha+\delta},
\hat{\mathfrak{g}}_{-\alpha}]$ for every $\alpha\in \Delta_-$
corresponding to a box lying to the north-east of $q$ and
by $[\hat{\mathfrak{g}}_{-\alpha+\delta},
\hat{\mathfrak{g}}_{\alpha}]$ for every $\alpha\in \Delta_+$
corresponding to a box lying to the north-east of $p$;
\item the space $V_4$ generated by 
$\hat{\mathfrak{g}}_{\alpha+\delta}$ for every 
$\alpha\in \Delta_+$ corresponding to a box lying to the north-east 
of $p'$;
\item the space $V_5$ generated by 
$\hat{\mathfrak{g}}_{\beta}$, $\beta\in D_i$, $i>1$, or
$\beta\in D_1$, $\beta\neq \gamma+\delta$, $\gamma\in\Delta_+$.
\end{itemize}
Similarly to the previous case one shows that $\mathfrak{j}$ 
is an ideal. By construction, $\Psi(\mathfrak{j})=(p,q,p',q')$. This proves 
our theorem for all quadruples  satisfying \eqref{eqn4}
and completes the proof.
\end{proof}

For small values of $n$ the number of equivalence classes of
support equivalent ideals for $\hat{\mathfrak{sl}}_n$ 
(obtained by a direct calculation) is as follows:
\begin{displaymath}
1,\,4,\,21,\,100,\,455,\dots 
\end{displaymath}
This sequence again seems to be new and does not appear in \cite{OEIS}.

\vspace{1cm}

\noindent
K.B.: Department of Mathematics, ETH Zurich, Raemistrasse 101, CH-8092 Zurich, Switzerland; e-mail: {\tt baur\symbol{64}math.ethz.ch}
\vspace{0.5cm}

\noindent
V.M.: Department of Mathematics, Uppsala University, 
Box 480, SE-751 06, Uppsala, Sweden; e-mail: {\tt mazor\symbol{64}math.uu.se}
\vspace{0.5cm}

\end{document}